\documentclass[a4paper,leqno,fleqn]{research-notes}
\pdfoutput=1

\usepackage[british]{babel}
\usepackage{easybmat}
\usepackage{xfrac}
\newcommand{\scr}[1]{{\scriptstyle{#1}}}
\newcommand{\rscr}[2]{{\raisebox{#2pt}{$\scriptstyle{#1}$}}}

\renewcommand{\mytitle}{Full Asymptotics and Laurent Series of Layer Potentials for Laplace's Equation on the Half-Space}
\newcommand{\myshorttitle}{Asymptotics and Laurent Series of Layer Potentials}
\renewcommand{\myauthor}{Karsten Fritzsch}
\renewcommand{\mymail}{k.fritzsch@math.uni-hannover.de}
\renewcommand{\myuni}{Leibniz Universit\"at Hannover, Institut f\"ur Analysis}

\hypersetup{pdfinfo={Title={\mytitle},
                     Keywords={PDE, Layer Potentials, Asymptotics, Manifolds with Corners},
                     Author={Karsten Fritzsch},
                     CreationDate={20171221130000+01'00'},
                     ModDate={20171221130000+01'00'}}}

\fancypagestyle{main}{
  \fancyhf{}
  \fancyhead[l]{\textsc{\myshorttitle}}
  \fancyhead[r]{\myauthor}

  \fancyfoot[c]{\small\textit{page \ \thepage}}
  \addtolength{\headheight}{1pt}}

\renewcommand{\myskip}{\medskip}

\newcommand{\textb}{\textit{b}}

\newcommand{\preb}[1]{\prescript{b}{}{#1}}
 \newcommand{\prebad}[2]{\preb{\hspace{-#2 pt}#1}}

\newcommand{\bdens}[1]{\tfrac{d#1}{#1}}
 
 \newcommand{\myfrac}[2]{{\raisebox{2pt}{\ensuremath{\scriptscriptstyle #1}}\hspace{-1.5pt}/\hspace{-1.5pt}\raisebox{-2pt}{\ensuremath{\scriptscriptstyle{#2}}}}}
\newcommand{\bfracdens}[2]{\tfrac{d\,\myfrac{#1}{#2}}{\myfrac{#1}{#2}}}
\newcommand{\bcup}{\mathbin{\overline{\cup}}}
 \DeclareMathOperator*{\Bcup}{\overline{\bigcup}}

\renewcommand{\Re}{\mathrm{Re}\,}

\newcommand{\Op}[1]{\mathrm{Op}\big(#1\big)}

\DeclareMathOperator{\codim}{codim}
\DeclareMathOperator{\nul}{null}
\DeclareMathOperator{\supp}{supp}

\newcommand{\calA}{\mathcal{A}}
\newcommand{\calE}{\mathcal{E}}
\newcommand{\calF}{\mathcal{F}}
\newcommand{\calG}{\mathcal{G}}
\newcommand{\calI}{\mathcal{I}}
\newcommand{\calK}{\mathcal{K}}

\newcommand{\calM}{\mathcal{M}}
\newcommand{\calN}{\mathcal{N}}
 \newcommand{\bcalN}{\preb{\mathcal{N}}}
\newcommand{\calP}{\mathcal{P}}
 \newcommand{\calPD}{{\calP_D}}

\newcommand{\calY}{\mathcal{Y}}

\newcommand{\olR}{\overline{\R^n}}

\newcommand{\bff}{\mathrm{bf}}
\newcommand{\df}{\mathrm{df}}

\newcommand{\bOmega}{\prescript{b}{}{\Omega}}
\newcommand{\bOmegah}{\prescript{b}{}{\Omega}^\frac{1}{2}}

\newcommand{\single}{\mathcal{S}\hspace{-2pt}\mathcal{L}}
\newcommand{\double}{\mathcal{D}\hspace{-1.9pt}\mathcal{L}}
\newcommand{\bsingle}{\prescript{b}{}{\hspace{-2pt}\single}}
\newcommand{\bdouble}{\prescript{b}{}{\hspace{-0.5pt}\double}}

\newcommand{\pr}{\mathrm{pr}}

\DeclareMathOperator{\dlim}{\partial-lim}

\begin{document}
\selectlanguage{british}
\pagestyle{plain}
\pagenumbering{alph}


\title{\mytitle}
\author{\myauthor\thanks{\myuni}
                 \thanks{partially funded by DFG grant 211072416 and EPSRC grant EP/K036696/1} \\
        \nolinkurl{k.fritzsch@math.uni-hannover.de}}
\date{\normalsize \today}
\maketitle
\thispagestyle{empty}


\begin{quote} \itshape We probe the application of the calculus of conormal distributions, in particular the Pull-Back and Push-Forward Theorems, to the method of layer potentials to solve the Dirichlet and Neumann problems on half-spaces. We obtain full asymptotic expansions for the solutions (provided these exist for the boundary data) and a new proof of the classical jump relations as well as Siegel and Talvila's growth estimates, using techniques that can be generalised to geometrically more complex settings. This is intended to be a first step to understanding the method of layer potentials in the setting of certain non-Lipschitz singularities and to applying a matching asymptotics ansatz to singular perturbations of related problems. \end{quote}


\pagestyle{main}

\section{Introduction}\label{sec:intro}
\setcounter{page}{1}\pagenumbering{arabic}

The method of layer potentials is a classical and well-studied approach to solving the Dirichlet and Neumann problems, in particular but not only for the Laplacian. If $Y \subset \R^n$ is a smooth and closed submanifold of dimension $n-1$ with connected interior domain $X$ (the bounded component of $\R^n \setminus Y$) and exterior unit normal field $\nu$, let $$\dlim u (y) = \lim\limits_{t \to 0^+} u(y-t\nu)$$ denote the pointwise limit to $y \in Y$ from within $X$, if it exists. Then, the (interior) Dirichlet and Neumann problems for $\Delta$,
 \begin{gather}
  \Delta u = 0 \quad\text{in $X$,} \qquad \phantom{\partial_\nu}\dlim u =f
   \label{intro:Dir} \\
  \Delta v = 0 \quad\text{in $X$,} \qquad\, \dlim \partial_\nu v =g
   \label{intro:Neu}
 \end{gather}
are reduced to solving certain Fredholm integral equations on $Y$. In detail, if $n \ge 3$ and if
 \[ P = \frac{|z-z^\p|^{2-n}}{(2-n)\vol\Sph^{n-1}} \]
denotes the fundamental solution for $\Delta$ on $\R^n$, then $\single = P \big|_{X \times Y}$ and $\double = \partial_\nu(z^\p) P \big|_{X \times Y}$ are the single and double layer potential kernels for $\Delta$ and $X$. The boundary integral equations then arise from the jump relations for the induced operators,
 \begin{align}
  \dlim \Op{\single}h &= \Op{\calN}h \quad, \label{intro:jump.1}\\
  \dlim \Op{\double}h &= \tfrac{1}{2}h + \Op{\calK}h \quad, \label{intro:jump.2}\\
  \dlim \partial_\nu \Op{\single}h &= -\tfrac{1}{2}h - \Op{\calK}^*h \quad, \label{intro:jump.3}
 \end{align}
where $\calN = P\big|_{Y^2} = \dlim \single$ and $\calK = \partial_\nu(z^\p)P\big|_{Y^2} = \dlim\double$ are the single and double boundary layer kernels. Good accounts of this can be found in \cite{Fol95,McL00,HW08}, for instance. Another important application of the method of layer potentials is the study of the Dirichlet-to-Neumann operator. If solution operators to \Cref{intro:Dir,intro:Neu} are constructed and their Fredholm and invertibility properties are established, then these can be used to construct the Dirichlet-to-Neumann operator and derive its properties, cf.\ for instance \cite[(7.33)]{MT99}.

In the past decades, this method has seen generalisations in various directions: Restricting the regularity of $Y$ or allowing $Y$ to be unbounded, to name but a few. The seminal works \cite{Cal77, CMM82} allowed to establish the classical method in the case of $Y$ being a Lipschitz hypersurface in a smooth and closed Riemannian manifold (see \cite{FJR78,Ver84} and \cite{MT99,MT00,MMT01} and the references given there), but this fails when $X$ has, e.g., cusps. Layer potentials for domains with corners have been studied in \cite{FJL77,TMV05} or \cite{ZM84,PS90,Els92,Rat92}, amongst others. If $Y$ is a coordinate plane (and hence $Y$ is unbounded and $X$ the half-space), the method has been extended in \cite{FS75,Arm76,Arm79,Gar81,Yos96}, for instance. In order to allow for polynomial growth of boundary data, the kernels have been modified by subtracting terms in an expansion in homogeneous harmonic functions, see \cite{Tal97} for a comprehensive account, yielding growth estimates for solutions (as in \cite{St96,Tal97,RY13,YR14,Qia15}) and existence and uniqueness for (inhomogeneous) Dirichlet and Neumann problems in weighted Sobolev spaces (cf. \cite{AN01,Amr02}).

The Dirichlet and Neumann problems and the method of layer potentials are also closely related to a specific transmission problem, the plasmonic eigenvalue problem, in which solutions are parameterised by reciprocals of eigenvalues of a combination of the exterior and interior Dirichlet-to-Neumann operator, cf. \cite{Gri14}. In order to understand the behaviour of solutions to the plasmonic eigenvalue problem in singular non-Lipschitz limits, one can use a matching asymptotics ansatz: A sequence of blow-ups (in the sense of \cite{Mel93}) is used to resolve the geometry to a manifold with corners and the existence of a solution with full asymptotic expansions at the boundary hypersurfaces is assumed. Matching up these expansions across transmission faces one obtains model problems which can then be studied separately, for instance by using the method of layer potentials.

\medskip This work is intended to be the first step to extend the method of layer potentials to singular domains that can be resolved in terms of manifolds with corners, including boundary data with polyhomogeneous behaviour near such singularities. (For instance, the cusp-like and non-Lipschitz singularity of the complement of two smooth and touching domains.) A function $\varphi : M \too \C$ on a manifold with boundary is called polyhomogeneous, if it is smooth in $\mathring{M}$ and near $\partial M$ admits an expansion
 \[ \varphi \sim \sum_{(\alpha,p) \in E} \varphi_{\alpha,p} \, \rho^\alpha \log^p \rho \]
in terms of a boundary defining function $\rho$, and where $E \subset \C \times \N$ is a suitable index set, see \Cref{APP.ssec:phgy} for more detail on this.  The natural framework for this is the calculus of conormal distributions of \cite{Mel92} and singular pseudodifferential calculi adapted to the specific geometry, for instance the $b$-calculus of \cite{Mel93}. Of particular importance are the Pull-Back and Push-Forward Theorems, cf. \Cref{APP.ssec:maps.fibrations}, as these allow us to show that the operators under consideration preserve polyhomogeneity. This scheme has been started in \cite{Fri}, including the situation of two touching domains. In this first part at hand, we consider layer potentials on the radial compactification of a half-space and reformulate the mapping properties and the classical jump relations in terms of polyhomogeneous functions. We obtain a new proof of the jump relations that relies on the Push-Forward Theorem and hence carries over easily to different geometries. We also consider modified layer potentials, in this case enabling us to include boundary data with polynomial growth at the face at infinity.

\paragraph{Main Results and Outline} In the following we always let $n \ge 3$. Let $X = \overline{\R^n_+}$ be the radially compactified half-space, $Y = \overline{\R^{n-1}}$ be the finite boundary of $X$ and $Z$ be the boundary at infinity of $X$. Restricting the construction of the $b$-double space of $\overline{\R^n}$ to the subspace $X \times Y$, and further resolving the boundary diagonal $\{(x,y) \in X \times Y \,|\, x=y \}$, in \Cref{sec:spaces.kernels} we obtain a manifold with corners $\calP_D$ and a canonical diffeomorphism of one of its faces with the $b$-double space of $Y$, $\lf(Y) \cong Y^2_b$. This space is constructed so that, in \Cref{ssec:layer.potentials}, we can show:

\begin{thm}\label{thm:intro.phg} The \textit{(modified)} layer potential kernels for $\Delta$, $\single_k$ and $\double_k$ as defined in \cref{def:single,def:double,def:mod.single,def:mod.double}, lift to define polyhomogeneous $b$-densities on $\calP_D$. They are smooth up to the face $\lf(Y)$ with restrictions given by the \textit{(modified)} boundary layer kernels $\calN_k$ and $\calK_k = 0$.
\end{thm}
Here, the subscript $k$, where $k \in \N$ (and $0 \in \N$), indicates a modified variant of the respective kernel for which the $k$ leading order terms at the boundary at infinity, $Z$, have been removed. Considering these is necessary in order to accommodate the possible growth of data at infinity (in $Y$): If $k > \pm 1 - \Re E$, choosing $+1$ and $-1$ for the single and double layer kernel, the integrability condition in the Push-Forward Theorem is exactly the usual integrability condition for the modified layer potentials.

To prove \Cref{thm:intro.phg}, we rely on local coordinate expressions for the fundamental solution of $\Delta$. Using the Pull-Back and Push-Forward Theorems, we see that the layer kernels define operators between spaces of polyhomogeneous functions and we use a simple formula for push-forwards, as in \cite{GG01,gohar}, to improve the index sets and obtain the expected mapping properties. This also allows us to construct examples showing that solutions with and without logarithmic terms at infinity may arise.

\begin{thm}\label{thm:intro.mapping} Let $E$ be an index set, $f \in \calA^E(Y)$ and $k_\pm = k_\pm(E) = \min\{\, l \in \N\,|\, l > \pm 1 - \Re E\,\}$. Then, $\Op{\single_{k_+}}f$ and $\Op{\double_{k_-}}f$ exist and are polyhomogeneous on $X$ with index families given in \cref{eq:potentials.phg}. The layer potentials of $f$ have Laurent expansions at the boundary at infinity $Z$ if and only if conditions \cref{eq:log.condition} respectively \cref{eq:log.condition.mod} hold.
\end{thm}
In \Cref{ssec:bdy.kernel}, we show that the (modified) boundary layer kernel $\calN_k$ defines an element of the full $b$-calculus and then go on to use the formula from \cite{gohar} again to show the following formulation of the classical jump relations, cf. \Cref{thm:formulae}:

\begin{thm}\label{thm:intro.jump} With respect to a suitable trivialisation of a neighbourhood of $Y \subset X$, there are functions $R^\pm_k \in \dot{C}^0\big([0,1);\calA(Y)\big)$ so that
 \begin{align*}
   \Op{\single_k}f(\chi,y) &= \Op{\calN_k}f(y) + R^+_k(\chi,y) \;, \\
   \Op{\double_k}f(\chi,y) &= \phantom{-} \tfrac{1}{2}f(y) + R_k^-(\chi,y) \;, \\
   \partial_{\nu}\big(\Op{\single_k}f\big)(\chi,y) &= -\tfrac{1}{2}f(y) - R_k^-(\chi,y) \;.
 \end{align*}
\end{thm}
(Here, $\dot{C}^0\big([0,1);\calA(Y)\big)$ denotes the space of continuous functions $[0,1) \too \calA(Y)$ that vanish at $\{0\}$.) Since solutions to \Cref{intro:Dir,intro:Neu} are given by layer potentials up to certain harmonic polynomials and because harmonic polynomials are clearly polyhomogeneous, we immediately obtain polyhomogeneity and hence full asymptotic expansions of solutions to the Dirichlet and Neumann problems:

\begin{thm}\label{thm:intro.solns} If the boundary data $f$ respectively $g$ is polyhomogeneous, $f$, $g \in \calA^E(Y)$, then the polyhomogeneous Dirichlet and Neumann problems
 \begin{align}
  \Delta u &= 0 \quad\text{on $X$}\,, & \dlim u &= f \,,
   &u \in \calA^\calE(X) & \;\,\text{with $\Re\calE(Z) \ge \min\{\Re E,n-1\}$,}
   \label{intro.eq:dir.phg} \\
  \Delta v &= 0 \quad\text{on $X$}\,, & \dlim \partial_\nu v &= g \,,
   &v \in \calA^\calF(X) & \;\text{with $\Re\calF(Z) \ge \min\{\Re E - 1,n-2\}$,}
   \label{intro.eq:neu.phg}
 \end{align}
have solutions. Solutions are given by the \textit{(modified)} double respectively single layer potentials and are unique up to the addition of harmonic polynomials of degree up to $k_\pm(E)$ that in the Dirichlet case vanish at $Y$ and in the Neumann case have normal derivative vanishing at $Y$.
\end{thm}
One consequence of using the calculus of conormal distributions is that from \Cref{thm:intro.jump} we obtain, without further work, the fact that the limits $\dlim u$ and $\dlim \partial_\nu v$ in \Cref{thm:intro.solns} are attained uniformly on compact subsets including all derivatives (of $u$ respectively $\partial_\nu v$) in directions tangent to $Y$.

\medskip As detailed above, in \Cref{sec:spaces.kernels} we construct the resolution and manifold with corners $\calPD$ and introduce kernels and $b$-densities on this space. In the main part, \Cref{sec:method}, we introduce the (modified) layer kernels, show that they define polyhomogeneous $b$-densities on $\calPD$, determine their mapping properties between spaces of polyhomogeneous functions and then give a new proof of the classical jump relations. This leads directly to \Cref{thm:intro.solns}. Most detailed calculations and proofs can be found in \Cref{APP.sec:calcs}, while \Cref{APP.sec:defs} contains background material on manifolds with corners, polyhomogeneity and the Pull-Back and Push-Forward Theorems.

\paragraph{Further Remarks} As mentioned before, this is but a first step and is meant to be a test case for the applicability of the calculus of conormal distributions to the method of layer potentials. In particular, it shows which challenges await: In the general case (when $Y = \partial X$ is not flat), the proof of the jump relations holds true (with minor modifications) but the boundary double layer potentials $\calK_k$ will not vanish but define nontrivial elements of the full $b$-calculus. Thus, in order to solve the Dirichlet and Neumann problems (while at the same time staying within the class of polyhomogeneous functions), we will need to extend results on Fredholm properties to the full $b$-calculus. The same is already true in the case of the half-space for the oblique derivative problem or the construction of the Dirichlet-to-Neumann operator as both of these involve the inversion of $\calN_k$.


\paragraph{Acknowledgements} The author would like to express his gratitude for the support and supervision he has enjoyed from his advisor Daniel Grieser. He would also like to thank Michael Singer, Chris Kottke and Richard B. Melrose for comments, suggestions and discussions.


\section{Spaces and Kernels}\label{sec:spaces.kernels}

Let $n \ge 3$ and $X=\overline{\R^n_+}$ denote the radial compactification of the half-space
 \begin{equation}\label{SK.eq:half.space}
  \R^n_+ = \myset{z = (x,y) \in \R \times \R^{n-1}}{x \ge 0} \;.
 \end{equation}
$X$ is a compact and smooth manifold with corners, having two boundary hypersurfaces: The boundary hypersurface at infinity, $Z$, which can be identified with a closed half-sphere of dimension $n-1$, and the boundary hypersurface at $x=0$, $Y$, which we will identify with $\overline{\R^{n-1}}$. They meet in a single corner, identifiable with $\Sph^{n-2}$ and denoted by $\Gamma$. We will make use of the following sets of adapted local coordinates:
 \begin{equation}\label{SK.eq:corner.coordinates}\begin{array}{cl}
  \big(x,y\big) &\text{ near the interior of $Y$,} \\
  \big(\rho,\vartheta\big) &\text{ near the interior of $Z$,} \\
  \big(xr,r,\omega\big) &\text{ near $\Gamma$.}
 \end{array}\end{equation}
Here $(\rho^{-1},\vartheta) \in \R_+ \times \Sph^{n-1}_+$ denote polar coordinates for $z$ and $(r^{-1},\omega) \in \R_+ \times \Sph^{n-2}$ polar coordinates for $y$. (Note that $\rho$ and $r$ denote the reciprocal modulus of $z$ respectively $y$.) The term \emph{adapted} refers to the fact that these sets contain boundary defining functions: $x$ and $\rho$ are boundary defining functions for $Y$ and $Z$, respectively (valid near the interior in each case), while $r$ and $xr$ are boundary defining functions for $Z$ and $Y$ near the corner $\Gamma$. $\vartheta$ and $\omega$ are coordinates in $Z$ and $\Gamma$, respectively.

When dealing with products of faces of $X$ (including $X$ itself), we will mark the coordinate in the second factor by a superscript prime as is common notation. And when writing index families for $X$, we will always refer to the ordering $\calM_1(X) = \big\{\, Y,\, Z\,\big\}$.

\subsection{The Poisson and Double Spaces}\label{ssec:spaces}

As the radial compactification $\overline{\R^n}$ has only one boundary component, the sphere at infinity $\Sph_\infty$, there is no ambiguity in defining the $b$-double space, it is given by blowing up the corner $\Sph_\infty \times \Sph_\infty$ in $\big(\overline{\R^n}\big)^2$:
 \begin{equation}
  \preb{\big(\overline{\R^n}\big)}^{\raisebox{2pt}{$\scriptstyle 2$}}
   = \Big[\big(\overline{\R^n}\big)^2;\big(\Sph_\infty\big)^2\Big] \xrightarrow{\quad\beta\quad} \overline{\R^n}^2 \;,
 \end{equation}
cf.~\Cref{APP.ssec:blow.ups}. Recall that this blow-up amounts to substituting the unit normal bundle to the corner $\Sph_\infty \times \Sph_\infty$ in $\big(\overline{\R^n}\big)^2$ for $\Sph_\infty \times \Sph_\infty$ itself and that $\beta$ is a diffeomorphism away from the corner. In this way, the newly created face (called the front face of the blow-up) encodes directions in which the corner can be approached from within.

The blow-down map $\beta$ has corresponding versions mapping into $X \times Y$ and $Y \times Y$. Let
 \begin{equation}\begin{array}{cccl}
  \calP \deq \big[X \times Y; Z \times \Gamma\big]
    &\xrightarrow{\quad\beta_\calP\quad} &X \times Y \;, \\
  Y^2_b = \big[Y^2 ; \Gamma^2 \big] &\xrightarrow{\quad\beta_Y\quad} &Y \times Y \;.
 \end{array}\end{equation}
$Y^2_b$ is the standard $b$-double space of $Y$. In fact, there are canonical isomorphisms $\calP \cong \beta^*(X \times Y)$ and $Y^2_b \cong \beta^*(Y^2)$, though we will not go into the details here. (Both $X \times Y$ and $Y \times Y$ are $b$-submanifolds of $\olR^2$ intersecting $\Sph_\infty^2$ freely, cf. \cite[V.8 ff.]{Mel96}.) Following the conventions in \cite{Loy98}, we denote the boundary hypersurfaces of $Y^2_b$ by
 \begin{equation}
  \lf_Y \deq \beta_Y^\ast\Big( \Gamma \times Y\Big) \;,\quad
  \rf_Y \deq \beta_Y^\ast\Big( Y \times \Gamma\Big) \;,\quad
  \bff_Y \deq \beta_Y^\ast\Big( \Gamma \times \Gamma\Big) \;.
 \end{equation}
(After blow-down, the left face $\lf_Y$ corresponds to the \qt{right factor,} $\beta_Y(\lf_Y) = \Gamma \times Y$ and vice versa for $\rf_Y$.) The four boundary hypersurfaces of $\calP$ will be denoted by
 \begin{equation}\begin{aligned}
  \rf_\calP &\deq \beta_\calP^\ast\Big(X \times \Gamma\Big) \;,
  &\lf_\calP(Y) &\deq \beta_\calP^\ast\Big(Y \times Y\Big) \;, \\
  \lf_\calP(Z) &\deq \beta_\calP^\ast\Big(Z \times Y\Big) \;,
  &\bff_\calP &\deq \beta_\calP^\ast \Big( Z \times \Gamma\Big) \;,
 \end{aligned}\end{equation}
cf.~\Cref{fig:poisson}. Similarly to the isomorphisms mentioned above, there is a canonic isomorphism $Y^2_b \cong \lf_\calP(Y)$. We will make use of the adapted coordinates in $\calPD$ derived from those in \Cref{SK.eq:corner.coordinates}, see the first paragraph of \Cref{APP.sec:calcs} for two examples.

\begin{figure}[ht]\centering
 \includegraphics[width=0.9\textwidth]{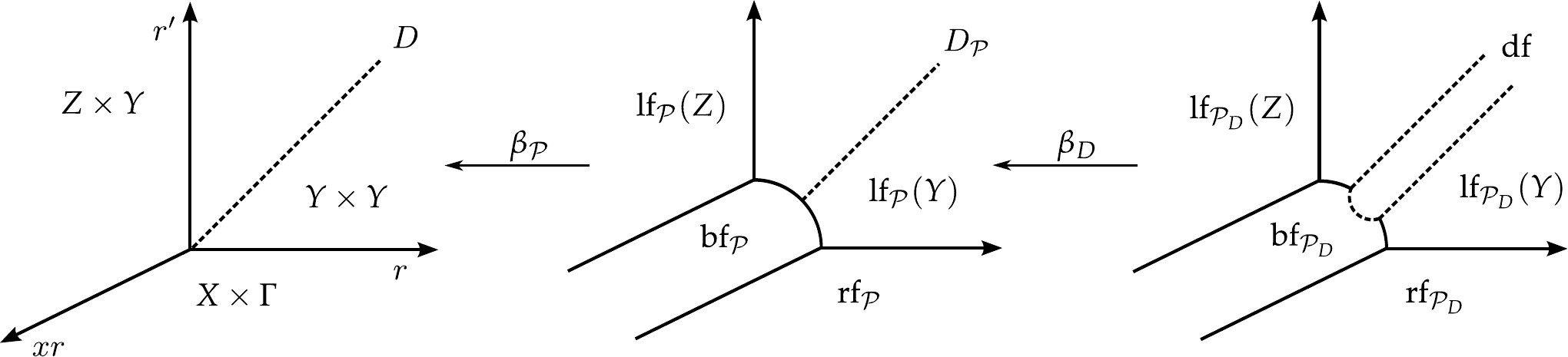}
 \caption[The Poisson space for the half-space]
         {The (resolved) Poisson space $\calP$ resp. $\calPD$ and its boundary hypersurfaces.}
 \label{fig:poisson}
\end{figure}

It will be useful to perform another blow-up in $\calP$, namely of the intersection of the lifted diagonal with $\lf_\calP(Y)$. The diagonal in $X \times Y$ is given by
 \begin{equation}\label{eq:diagonal}
  D \deq \myset{(z,z^\p) \in X \times Y}{z = z^\p}
 \end{equation}
and it lifts to a $p$-submanifold $D(\calP) \deq \beta_\calP^*(D) \subset \calP$, which can be readily derived from the fact that the diagonal in $\olR^2$ lifts to a $p$-submanifold of the respective $b$-double space. Then, its intersection with $\lf_\calP(Y)$ (another $p$-submanifold) is a $p$-submanifold as well and we define the \emph{resolved Poisson space} to be
 \begin{equation}
  \calPD \deq \Big[ \calP ; D(\calP) \cap \calY\Big] \xrightarrow{\quad \beta_\calPD \quad} \calP \;.
 \end{equation}
We let $\beta_{\calPD} \deq \beta_D \circ \beta_\calP$ and denote the lifts of \qt{old faces} as before, though we might equip them with a subscript $\calPD$ to emphasise the space. The newly created face, the front face of $\beta_D$, will be denoted by $\df$, as in \emph{diagonal face}.

When referring to index families for these spaces, we will always use the following ordering of boundary hypersurfaces:
 \begin{equation}\label{eq:bhs.orderings}\begin{aligned}
  \calM_1(Y^2_b)  &= \big\{\, \lf_Y,\, \rf_Y,\, \bff_Y \, \big\} \\
  \calM_1(\calP)  &= \big\{\, \lf_\calP(Y),\, \lf_\calP(Z),\, \rf_\calP,\, \bff_\calP \, \big\} \\
  \calM_1(\calPD) &= \big\{\, \lf_\calPD(Y),\, \lf_\calPD(Z),\, \rf_\calPD,\, \bff_\calPD,\, \df \, \big\}
 \end{aligned}\end{equation}

\subsection{Kernels and \textb-Densities}

In order to apply the Push-Forward Theorem, see \Cref{APP.ssec:maps.fibrations}, as well as related results on mapping properties between weighted Sobolev spaces and spaces of polyhomogeneous functions (as in \Cref{thm:layer.ops.phg,thm:bdy.mapping}), we will need to interpret distributional kernels as $b$-densities on $\calPD$. We do this by choosing two $b$-densities on $X$ respectively $Y$: On $X$, let $\mu = x^{-1} |dz|$, while on $Y$ we use $|dz^\p|$.  Then, $\mu \in \calA^{(0,1-n)}\big(X;\bOmega\big)$ and $|dz^\p| \in \calA^{1-n}\big(Y;\bOmega\big)$ are polyhomogeneous $b$-densities, as we show in \Cref{APP.thm:dens.phg}. Let $\pr_l$, $\pr_r$ denote the projections \emph{onto} the left respectively right factor of $X \times Y$ and define $\pi_\bullet \deq \pr_\bullet \circ \beta_{\calPD}$ to be the corresponding blow-down projections. Again, be aware that, e.g., $\pi_l$ projects \emph{off} the left faces in the sense that $\pi_l\big(\lf_\calPD(H)\big) = H$ for $H = Y$, $Z$. The blow-down projections are $b$-fibrations with exponent matrices of a simple form as we show in \Cref{APP.thm:b.exponents}. Now, given a distributional kernel $A$ on $X \times Y$, let
\begin{equation}\label{eq:kernel.density.ident}
  \prebad{A}{2}
   \deq \beta_\calPD^\ast\big(A\big) \, \pi_l^\ast\big(\mu) \, \pi_r^\ast\big(|dz^\p|\big)
 \end{equation}
denote the associated density. The pull-backs of these $b$-densities are well-defined, compare the remark at the end of \Cref{APP.ssec:blow.ups}. Classically, $A$ defines an operator $\Op{A}$ via
 \begin{equation}\label{eq:plain.op}
  \Op{A}f(z) \deq \int_Y A(z,z^\p)f(z^\p)\,|dz^\p| \;,
 \end{equation}
whenever the integral converges and this can be interpreted as a combination of pull-backs and push-forwards,
 \begin{equation}\label{eq:density.op}
  \Op{A}f(z)\mu_l = (\pi_l)_*\Big( \prebad{A}{2} \; \pi_r^*(f) \Big) \;.
 \end{equation}
As hinted at by the left hand sides of \cref{eq:plain.op} and \cref{eq:density.op}, we then identify functions with $b$-densities by multiplication by $\mu$ respectively $|dz^\p|$,
 \begin{equation}\label{eq:fct.dens.ident}
  \calA\big(X;\bOmega\big) \longleftrightarrow \mu \calA(X) \quad\text{ and }\quad
  \calA\big(Y;\bOmega\big) \longleftrightarrow |dz^\p| \calA(Y) \;.
 \end{equation}
Observe that these identifications change the index families, though: In detail, we identify
 \begin{equation}\label{eq:fct.dens.ident.detail.1}
  \calA^{(E,F)}(X) \ni u \longleftrightarrow u\mu \in \calA^{(E,F+1-n)}\big(X;\bOmega\big)
 \end{equation}
and
 \begin{equation}\label{eq:fct.dens.ident.detail.2}
  \calA^{E}(Y) \ni f \longleftrightarrow f|dz^\p| \in \calA^{E + 1 - n}\big(Y;\bOmega\big) \;.
 \end{equation}
Then, as a short hand we write
 \begin{equation}\label{eq:final.op.ident}
  \Op{A}f = (\pi_l)_*\Big(\prebad{A}{2}\, (\pi_r)^* f\Big) \;,
 \end{equation}
and usually understand the left-hand side as a function. Similarly, we will identify kernels on $Y \times Y$ (or rather $Y^2_b$) with $b$-densities by
 \begin{equation}\label{eq:kernel.density.ident.bdy}
  \preb{B} \deq \beta_Y^*(B) \, (\pi_l^\partial)^\ast(|dy|) \, (\pi_r^\partial)^\ast(|dy^\p|) \;,
 \end{equation}
where $\pi_\bullet^\partial \deq \pr_\bullet^\partial \circ \beta_Y$ and $\pr_\bullet^\partial : Y^2 \too Y$ are the projections \emph{onto} the left respectively right factor and we have used the coordinates $y$ and $y^\p$ to emphasise the fact that these are defined on $Y = \overline{\R^{n-1}}$.

\section{The Method of Layer Potentials}\label{sec:method}

The starting point for the method of layer potentials for the Laplacian on the half-space is its fundamental solution
 \begin{equation}\label{eq:fund.sol}
  P(z,z^\p) \deq \frac{|z-z^\p|^{2-n}}{(2-n)\vol \Sph^{n-1}} \;,
 \end{equation}
which is smooth for $z$, $z^\p \in \R^n$, $z \neq z^\p$ and has a conormal singularity at the diagonal. Its restriction to $\R^n_+ \times \R^{n-1}$ defines the \emph{Neumann kernel}
 \begin{equation}\label{def:single}
  \single \deq P\sub{\R^n_+ \times \R^{n-1}} = \frac{|z-z^\p|^{2-n}}{(2-n)\vol \Sph^{n-1}}
   = \frac{ (x^2 + |y-y^\p|^2)^{\tfrac{2-n}{2}}}{(2-n)\vol\Sph^{n-1}}
 \end{equation}
and the restriction of its normal derivative defines the \emph{Poisson kernel}
 \begin{equation}\label{def:double}
  \double \deq \big(\partial_\nu(z^\p) P\big)\sub{\R^n_+ \times \R^{n-1}}
   = -\frac{|z-z^\p|^{-n}}{\vol \Sph^{n-1}} \scal{\nu(z^\p)}{z-z^\p}
   = x \frac{ (x^2 + |y-y^\p|^2)^{-\tfrac{n}{2}}}{\vol\Sph^{n-1}} \;.
 \end{equation}
(Here and in the following, $\nu$ will denote the lift of the vector field $-\partial_x \in C^\infty\big(\R^n_+;T\R^n_+\big)$ to $X$ or one of the double spaces.) These kernels will also be called the \emph{single} and \emph{double layer kernels}. Given a function $f : Y \too \R$, by means of \cref{eq:plain.op}, the kernels define operators
 \begin{align}
  \Op{\single}f (z) &\deq \int_Y \single(z,z^\p) f(z^\p) \big|dz^\p\big|
   = (\pi_l)_*\Big(\bsingle\,(\pi_r)^*f\Big) (z) \;,
   \label{eq:plain.single.op} \\
  \Op{\double}f (z) &\deq \int_Y \double(z,z^\p) f(z^\p) \big|dz^\p\big|
   = (\pi_l)_*\Big(\bdouble\,(\pi_r)^*f\Big) (z) \;,
   \label{eq:plain.double.op}
 \end{align}
whenever the integrals converge. We identify \cref{eq:plain.single.op,eq:plain.double.op} with functions on $X$ via \cref{eq:fct.dens.ident.detail.1} and call the resulting operators \emph{single} and \emph{double layer operators}.

\myskip As mentioned earlier, \cite{FS75} and later \cite{Gar81} introduced modified layer kernels so as to be able to study their action on functions of polynomial growth. The following construction is given (and studied) in detail in \cite{Tal97} as well. Consider the multipole expansion
 \begin{equation}\label{eq:multipole}
  \big|z-z^\p\big|^{-s}
   = \sum_{m=0}^\infty \frac{|z|^m}{|z^\p|^{m+s}} C^\frac{s}{2}_m(\Theta) \;,
  \end{equation}
where $|z| < |z^\p| \neq 0$, $\Theta = \tfrac{<z,z^\p>}{|z||z^\p|}$ and $C^s_m$ denotes the \emph{ultraspherical} or \emph{Gegenbauer polynomial}, see \cite[pp.~80]{Sze39}, of orders $m$ and $s$. If $|z| < |z^\p|$, the right hand side of \cref{eq:multipole} converges uniformly. \cref{eq:multipole} allows us to expand the layer potential kernels in terms of homogeneous, harmonic polynomials and this is the basis for the extension of the method of layer potentials to functions of polynomial growth. In the following, for convenience we define $C^s_{-1} = 0$.

\begin{prop}\label{thm:hom.harm.decomp} Let $z$, $z^\p \in \R^n$, with $|z| < |z^\p|$ and $\Theta = \tfrac{<z,z^\p>}{|z||z^\p|}$.
\begin{enumerate}
 \item In the multipole expansion
        \[ |z-z^\p|^{2-n} = \sum_{m=0}^\infty \frac{|z|^m}{|z^\p|^{m+n-2}} C^{\frac{n-2}{2}}_m(\Theta) \;, \]
     each summand on the right hand side is homogeneous \textit{(in both $z$ and $z^\p$) and harmonic (in $z$)}.
 \item In the multipole expansion
        \begin{align*}
         |z-z^\p|^{-n}\scal{\nu(z^\p)}{z-z^\p} &= \sum_{m=0}^\infty \Scal{\nu(z^\p)}{C_{n,m}(z,z^\p)} \\
          &= \sum_{m=0}^\infty x \frac{|z|^{m-1}}{|z^\p|^{m+n-1}} C^{\frac{n}{2}}_{m-1} (\Theta) \;,
        \end{align*}
     where $C_{n,m}(z,z^\p) = z \tfrac{|z|^{m-1}}{|z^\p|^{m+n-1}} C^{\frac{n}{2}}_{m-1}(\Theta) - z^\p \tfrac{|z|^m}{|z^\p|^{m+n}}C^{\frac{n}{2}}_{m}(\Theta)$, each summand on the right hand sides is homogeneous \textit{(in both $z$ and $z^\p$) and harmonic (in $z$)}.
\end{enumerate}\end{prop}

\begin{proof} A simple direct proof for this is given in \cite[Lemma 3.2.1]{Tal97}, also compare the references given there (namely \cite{Arm79,Gar81,Yos96}). The idea is the following: Fix $z^\p \in \R^{n-1}$. Using the uniform convergence of the multipole expansion \cref{eq:multipole}, harmonicity of $|z-z^\p|^{2-n}$ and the fact that the Laplacian preserves homogeneity, we see that the summands in item \emph*{i)} are harmonic on $\{|z| < |z^\p|\} \subset \R^n$. As they are homogeneous in $z$, they are indeed harmonic on all of $\R^n$. For the terms in the expansion of $|z-z^\p|^{-n}\scal{\nu(z^\p)}{z-z^\p}$, the same can be achieved by reordering the terms with respect to their homogeneity in $z$ and noting that the left hand side is still harmonic.
\end{proof}

Now let $\psi : \R^{n-1} \too \R_+$ be a smooth cut-off function vanishing for $|z^\p| < 1$ and being identically $1$ for $|z^\p|>2$. Then, for $k>0$ we define
 \begin{align}
  \single_k(z,z^\p) &\deq \single(z,z^\p)
    - \tfrac{1}{(2-n)\vol\Sph^{n-1}}\psi(z^\p)
        \sum_{m=0}^{k-1} \frac{|z|^m}{|z^\p|^{m+n-2}} C^{\frac{n-2}{2}}_m(\Theta) \;,
  \label{def:mod.single} \\
  \double_k(z,z^\p) &\deq \double(z,z^\p)
    + \tfrac{1}{\vol \Sph^{n-1}}\psi(z^\p) \sum_{m=0}^{k}\scal{\nu(z^\p)}{C_{n,m}(z,z^\p)}
  \label{def:mod.double} \\
   &\,= \double(z,z^\p) + \tfrac{1}{\vol \Sph^{n-1}}\psi(z^\p) \sum_{m=0}^{k-1} x \frac{|z|^m}{|z^\p|^{m+n}}C^{\frac{n}{2}}_m(\Theta) \;, \notag
 \end{align}
where again $(z,z^\p) \in \R^n_+ \times \R^{n-1}$, $z \neq z^\p$. The kernels $\single_k$ and $\double_k$ will be called the \emph{modified single} and \emph{double layer kernels}. In particular, and this is the essential part, we have removed the first $k$ summands in the expansions of the layer kernels at $|z^\p|=\infty$ in terms of homogeneous harmonic functions. For notational simplicity, we also write
 \begin{equation}\label{eq:zero.layers}
  \single_0 \deq \single \quad\text{ and }\quad \double_0 \deq \double \;.
 \end{equation}

The second set of layer kernels is given by further restriction to $z \in \R^{n-1}$. For $k \ge 0$ we define the \emph{(modified) single boundary layer kernels} by
 \begin{equation}\label{def:bdy.single.kernel}\begin{aligned}
  \calN_k &\deq \dlim \single_k = \single_k\sub{\R^{n-1}\times\R^{n-1}} \\
   &\,= \single(y,y^\p) - \tfrac{1}{(2-n)\vol\Sph^{n-1}}\psi(y^\p)
        \sum_{m=0}^{k-1} \frac{|y|^m}{|y^\p|^{m+n-2}} C^{\frac{n-2}{2}}_m(\Theta) \;,
 \end{aligned}\end{equation}
where $(z,z^\p) = (y,y^\p) \in \R^{n-1}\times\R^{n-1}$. In general, there is a \emph{(modified) double boundary layer kernel} $\calK_k$ as well. But one of the many simplifications in the situation of the half-space is that $\calK_k$ vanishes identically,
 \begin{equation}\label{def:bdy.double.kernel}
  \calK_k \deq \dlim \double_k = \double_k\sub{\R^{n-1}\times\R^{n-1}} = 0 \;.
 \end{equation}
Similarly to \cref{eq:plain.op}, the $\calN_k$ define operators, this time mapping functions on $Y$ to functions on $Y$,
 \[ \Op{\calN_k}f(z) \deq \int_Y \calN_k(z,z^\p) f(z^\p) \big|dz^\p\big| = (\pi^\partial_l)_*\Big(\preb{\calN_k} \, (\pi^\partial_r)^*f \Big)(z) \;, \]
where we make use of \cref{eq:fct.dens.ident.detail.2} to identify the right-hand side with a function on $Y$.

\myskip The (boundary) layer kernels define distributions on the respective double spaces, with singular supports contained in the respective diagonal. Following the philosophy of Melrose, as for instance described in \cite{Mel93}, we will lift these kernels to the resolved spaces constructed above and by doing so separate the singular behaviour of the kernels near boundary faces from that near the diagonal.

\subsection{The Layer Kernels}\label{ssec:layer.potentials}

Now, we will show that the (modified) layer kernels give rise to continuous maps between spaces of polyhomogeneous functions and study the resulting index families in detail. In particular, we will give examples of functions both giving rise to and \emph{not} giving rise to logarithmic behaviour of their layer potentials at $Z$.

\paragraph{Polyhomogeneity} Using simple calculations in local coordinates, see \Cref{APP.ssec:poly.calc}, we may show that the building blocks that make up the (modified) layer kernels lift to be polyhomogeneous functions on $\calPD$. Then, we identify these with $b$-densities via \cref{eq:kernel.density.ident}, dubbed the \emph{layer densities}, and determine their index families. Combined, we arrive at the following:

\begin{prop}\label{thm:kernels.phg} The \textit{(modified)} layer densities, defined via \cref{eq:kernel.density.ident}, lift to be polyhomogeneous $b$-densities on $\calPD$. More precisely:
 \begin{equation}\label{eq:kernels.phg}\begin{aligned}
  \bsingle &\in \calA^{(0,-1,-1,-n,1)}\big(\calPD,\bOmega\big) \qquad
  &\bdouble &\in \calA^{(1,0,1,1-n,0)}\big(\calPD,\bOmega\big) \\
  \bsingle_k &\in \calA^{(0,2-n-k,k-1,-n,1)}\big(\calPD,\bOmega\big) \qquad
  &\bdouble_k &\in \calA^{(1,1-n-k,k+1,1-n,0)}\big(\calPD,\bOmega\big)
 \end{aligned}\end{equation}
where the index families are written with respect to the ordering $\big(\lf(Y),\lf(Z),\rf(\Gamma),\bff,\df\big)$ of boundary hypersurfaces of $\calPD$.
\end{prop}

\begin{proof} We start by considering the unmodified kernels. Using \Cref{APP.thm:x.phg,APP.thm:abs.phg} and the structure of spaces of polyhomogeneous functions, we see that the layer kernels lift to be polyhomogeneous on $\calPD$:
 \begin{equation}\label{eq:kernels.phg.pre}
  \beta_{\calPD}^*(\single) \in \calA^{(0,n-2,n-2,n-2,2-n)}\big(\calPD\big) \quad,\qquad
  \beta_{\calPD}^*(\double) \in \calA^{(1,n-1,n,n-1,1-n)}\big(\calPD\big) \;,
 \end{equation}
in each case with respect to the ordering $(\lf(Y), \lf(Z), \rf, \bff, \df)$. To obtain the $b$-densities $\bsingle$ and $\bdouble$, we multiply by $\pi_l^*(\mu)\pi_r^*(|dz^\p|)$, which by \Cref{APP.thm:dens.phg} is polyhomogeneous with index family $(0,1-n,1-n,2-2n,n-1)$. Combining these, we arrive at the first line of \cref{eq:kernels.phg}.

Regarding the modified kernels, note that $\rf = \rf_\calP(\Gamma)$ is the only face having a neighbourhood on which the expansions from \Cref{thm:hom.harm.decomp} converge. There, they take the form
 \begin{align}
    \single_k &= \tfrac{1}{(2-n)\vol \Sph^{n-1}} \,
                \sum_{m=k}^\infty \frac{|z|^m}{|z^\p|^{m+n-2}} C^\frac{n-2}{2}_m(\Theta)
    \quad,\qquad \label{eq:mod.at.infty.1}\\
    \double_k &= -\tfrac{1}{\vol\Sph^{n-1}} \,
                \sum_{m=k}^{\infty} x \frac{|z|^{m}}{|z^\p|^{m+n}} C^\frac{n}{2}_{m}(\Theta)
    \quad. \label{eq:mod.at.infty.2}
 \end{align}
By \Cref{APP.thm:mod.phg,APP.thm:dens.phg}, we get index families $(0,\bullet,k-1,-n,\bullet)$ for $\bsingle_k$ and $(1,\bullet,k+1,1-n,\bullet)$ for $\bdouble_k$, where a $\bullet$ denotes an index set we have not determined yet, as the corresponding faces do not intersect $\rf$. Away from $\rf$, we use \cref{def:mod.single}, \cref{def:mod.double} and \Cref{APP.thm:mod.phg,APP.thm:dens.phg} as well as \cref{eq:kernels.phg.pre}. These give index families $(0,2-n-k,\bullet,-n,1)$ for $\bsingle_k$ and $(1,1-n-k,\bullet,1-n,0)$ for $\bdouble_k$. Altogether, we obtain the second line of \cref{eq:kernels.phg}.
\end{proof}

Having \Cref{thm:kernels.phg} at hand, the Push-Forward and Pull-Back Theorems, \Cref{APP.thm:pbt,APP.thm:pft}, provide bounds on the index sets of polyhomogeneous functions for their layer potentials to be defined. Moreover, these theorems prove polyhomogeneity of these layer potentials. But before we state this more precisely and determine the index families, observe the following: The Push-Forward Theorem does in general not give optimal index families for a specific combination of $b$-density and $b$-fibration. In particular, it allows for logarithmic terms that might in fact not arise. We will use the following two lemmata to improve the index families given by the Push-Forward Theorem in our special case.

\paragraph{Appearance of Logarithmic Terms} In order to determine (or improve) index sets, we need to determine the singular asymptotics of certain integrals. This is approached in different ways in \cite{Mel92} and \cite{bs.sal}, in \cite{GG01} both approaches are reviewed and compared. Explicit formulae for the coefficients can be obtained from \cite{bs.sal} for instance, and this is carried out in one important example in \cite{gohar}: Denote by $(x,y)$ and $t$ coordinates in $\R_+^2$ respectively $\R_+$ and let $F : \R_+^2 \too \R_+$, $F(x,y) = x y$. This is a $b$-fibration with exponent matrix $(1,1)^T$. Now if $u = \widetilde{u}(x,y) \bdens{x}\bdens{y}$ is a compactly supported $b$-density on $\R_+^2$, polyhomogeneous with respect to an integer index family $(k,l)$,
 \[ u \sim \sum_{i \ge k, j \ge l} x^i y^j \,u_{ij}\, \bdens{x}\bdens{y} \;, \]
then the push-forward of $u$ along $F$,
 \begin{equation}\label{eq:pft.sal.ex}
  ( F_* u)(t) = \left(\int_0^\infty \widetilde{u}(x,\tfrac{t}{x}) \bdens{x} \right) \bdens{t} \;,
 \end{equation}
is a polyhomogeneous $b$-density on $\R_+$, with respect to the index set $k \bcup l$, as follows from the Push-Forward Theorem, for instance. From \cite{gohar}, we obtain the following:

\begin{lemma}\label{thm:logs.exa} Let $F$ and $u$ be as above with $k$, $l \in \Z$. Then,
 \[ (F_*u)(t) \sim \sum_{j \ge \min\{k,l\}} t^j \big( a_j + b_j \log t \big) \;, \]
where the coefficients $b_j$ satisfy $b_j = - u_{jj}$, and $F_*u$ is polyhomogeneous with respect to the integer index set $k \cup l$ if and only if $u_{jj} = 0$ for all $j$. If $k$, $l \ge 0$, the leading order coefficient $a_0$ is given by
 \[ a_0 = \int_{\R_+} \widetilde{u}\big|_{y=0}\tfrac{dx}{x} + \int_{\R_+} \widetilde{u}\big|_{x=0} \tfrac{dy}{y} \;. \]
\end{lemma}
There are similar formulae for the coefficients $a_j$, $j>0$, as well as in the general case $k$, $l \in \Z$, but we will not make use of these. \Cref{thm:logs.exa} shows that the logarithmic terms in the expansion of the push-forward $F_*u$ vanish if and only if the \emph{diagonal elements} $u_{jj} = -b_j$ vanish for all $j \in \Z$. For the unmodified layer potential densities, we will use the following criterion to ensure this, its proof is given in \Cref{APP.sec:logs}.

\begin{lemma}\label{thm:log.lemma} Let $N$ be a manifold with corners and
 \[ F : [0,1)^2 \times N \times \Sph^m \too [0,1) \times N \quad,\qquad
     (x_1,x_2,y,\theta) \mapstoo (x_1 x_2, y) \;. \]
Suppose that $u \in \calA^{\calE}\big([0,1)^2 \times N \times \Sph^m;\bOmega\big)$ has compact support and that $E_1 \deq \calE(\{x_1=0\})$ and $E_2 \deq \calE(\{x_2=0\})$ are integer index sets. If
 \begin{equation}\label{eq:log.cond.1.change} u(x_1,x_2,y,\theta) = -u(-x_1,-x_2,y,-\theta)
 \end{equation}
for all $(x_1,x_2,y,\theta)$, then $F_*u$ is polyhomogeneous at the face $\{0\} \times N$ with respect to the integer index set $E_1 \cup E_2$.
\end{lemma}

\paragraph{Mapping Properties} Using \Cref{thm:logs.exa,thm:log.lemma}, in \Cref{thm:smooth.x,thm:smooth.x.mod} we show that the (modified) layer potential kernels give rise to functions that are smooth up to the boundary at $x=0$. Having thus improved the index sets at $Y$, we formulate the mapping properties of the layer potential operators between spaces of polyhomogeneous functions. As this is a direct consequence of the Push-Forward Theorem and \Cref{thm:smooth.x,thm:smooth.x.mod}, we do not give a proof but briefly sketch the calculus of index sets: As $\mathrm{null}(e_{\pi_l}) = \rf(Z)$, the integrability condition for $f \in \calA^{E}(Y)$ reads $\Re(E) > \pm 1 - k \eqd \alpha_\pm(k)$, where we take $+1$ for the single layer potentials and $-1$ for the double layer potentials, i.e.,
 \begin{equation}\label{def:alpha.k}
  \alpha_\pm(k) \deq \begin{cases} 1 - k &\text{for single layer potentials,} \\ -1-k &\text{for double layer potentials} \end{cases}
 \end{equation}
and use $k=0$ for the unmodified potentials. Observe that then, given $E$, the smallest $k$ for which the single/double layer potentials are defined on $\calA^E(Y)$ is given by the minimal $k \in \N$ so that $k > \pm 1 - \Re E$, that is by
 \begin{equation}\label{eq:min.k}
   k_\pm(E) \deq \min\{\, l \in \N \,|\, l > \pm 1 - \Re E\,\}
            = \max\{\, 0, \pm 1 + \lceil -\Re E \rceil \,\} \;,
 \end{equation}
where $\lceil t \rceil$ denotes the smallest integer greater than $t$. After using the Push-Forward Theorem to determine the index sets for the layer potentials of $f$, bear in mind that we still need to identify the resulting $b$-density with a function via \cref{eq:fct.dens.ident.detail.1} (though this does not change the index families).

\begin{thm}\label{thm:layer.ops.phg} Let $E$ be any index set for $Y$ and $f \in \calA^E(Y)$. If $E$ satisfies $\Re E > \alpha_\pm(k)$, the (modified) single and double layer potentials of $f$ are polyhomogeneous on $X$:
 \begin{equation}\label{eq:potentials.phg}\begin{aligned}
    \Op{\single}f &\in \calA^{(0,(E-1)\bcup(n-2))}(X) \\
    \Op{\double}f &\in \calA^{(0,E\bcup (n-1)}(X) \\
    \Op{\single_k}f &\in \calA^{(0,(E-1)\bcup(1-k))}(X) \\
    \Op{\double_k}f &\in \calA^{(0,E\bcup (-k))}(X)
 \end{aligned}\end{equation}
where we use the ordering $(Y,Z)$ for boundary hypersurfaces of $X$ and the identification \emph{\cref{eq:fct.dens.ident.detail.1}}.
\end{thm}
For instance, in case we have $f \in \calA^{l}(Y)$ for $l \in \Z$ and $l < 0$, we may use $k_\pm(l) = \pm 1 + 1 - l$ and deduce that
 \[ \Op{\single_{2-l}}f \in \calA^{(0,(l-1)\bcup(l-1))}(X) \quad\text{ and }\quad
    \Op{\double_{-l}}f \in \calA^{(0,l\bcup l)}(X) \;. \]
That is, the layer potentials of $f$ grow like $|z|^{1-l} \log |z|$ respectively $|z|^{-l} \log |z|$ near the face $Z$.

\paragraph{Examples} Thus, we have been able to exclude the appearance of logarithmic terms at the face $Y \subset X$ but not at the face at infinity, $Z$. The latter will in fact not be possible as for arbitrary data $f \in \calA^E(Y)$, the layer potentials will in general have logarithmic terms in their expansions near $Z$. In the case that $E$ is an integer index set, we will now use \Cref{thm:logs.exa} to produce a necessary and sufficient condition for the absence of logarithmic terms and use this to produce two examples.

\begin{prop}\label{thm:logs.condition} Let $f \in \calA^{l}(Y)$ for $l \ge \alpha_\pm(0)$. The unmodified layer potentials of $f$ are polyhomogeneous with respect to an integer index set at $Z$ if and only if
 \begin{equation}\label{eq:log.condition}
  \int_{\Sph^{n-2}} f_j(\omega^\p)
    C^{\frac{\bullet}{2}}_{j+1-n} (\scal{\theta}{\omega^\p})
     d\omega^\p = 0
  \quad \text{for $j \ge n-1$,}
 \end{equation}
for all $\theta \in \Sph^{n-1}_+$ and where a $\bullet$ stands for $n-2$ respectively $n$ in the case of the single respectively double layer potential and the $f_j$ are given by the polyhomogeneous expansion of $f$ at $\Gamma$, $f \sim \sum_{j \ge l} (r^\p)^j f_j(\omega^\p)$.
\end{prop}

\begin{prop}\label{thm:logs.condition.mod} Let $f \in \calA^l(Y)$, for some $l \ge \alpha_\pm(k)$. The modified layer potentials of $f$, $\Op{\single_k}f$ and $\Op{\double_k}f$, are polyhomogeneous with respect to an integer index set at $Z$ if and only if
 \begin{equation}\label{eq:log.condition.mod}\begin{aligned}
  \int_{\Sph^{n-2}} f_j(\omega^\p)
    C^{\frac{\bullet}{2}}_{j+1-n} \big(\scal{\theta}{\omega^\p}\big)
     d\omega^\p &= 0
  &&\quad \text{for $j \ge n-1$ \; and} \\[2ex]
  \int_{\Sph^{n-2}} f_j(\omega^\p)
    C^{\frac{\bullet}{2}}_{\alpha_\pm(0)-j} \big(\scal{\theta}{\omega^\p}\big)
     d\omega^\p &= 0
  &&\quad \text{for $\alpha_\pm(k)+1 \le j \le \alpha_\pm(k) + k$,}
 \end{aligned}\end{equation}
for all $\theta \in \Sph^{n-1}_+$ and where a $\bullet$ stands for $n-2$ respectively $n$ in the case of the single respectively double layer potential and the $f_j$ are given by the polyhomogeneous expansion of $f$ at $\Gamma$, $f \sim \sum_{j \ge l} (r^\p)^j f_j(\omega^\p)$.
\end{prop}

Please find proofs for \Cref{thm:logs.condition,thm:logs.condition.mod} at the end of \Cref{APP.sec:logs}. Using \cref{eq:log.condition,eq:log.condition.mod}, it is easy to derive explicit examples showcasing the appearance (and failure to appear) of additional logarithmic terms.

\begin{exa}\label{ex:logs} Let $n = 3$ and consider the functions
 \begin{equation}
  f(y^\p) = \frac{1}{1+|y^\p|^2} \quad\text{ and }\quad g(y^\p) = \frac{|y^\p|}{1+|y^\p|^4} \;.
 \end{equation}
Both $f$ and $g$ are smooth on $\mathring{Y} = \R^{n-1}$ and in particular, as
 \begin{equation}
  f(y^\p) = (r^\p)^2 \sum_{j=0}^\infty (-1)^j (r^\p)^{2j}
   \quad\text{ and }\quad
  g(y^\p) = (r^\p)^3 \sum_{j=0}^\infty (-1)^j (r^\p)^{4j}
 \end{equation}
for $|y^\p| > 1$, we have $f \in \calA^2(Y)$, $g \in \calA^3(Y)$. Thus, as $2 > \alpha_\pm(0) = \pm 1$, the unmodified single and double layer potentials of both $f$ and $g$ are well-defined. Since $C^{\frac{\bullet}{2}}_0 = 1$ and $f_2 = 1$, we have
 \begin{equation}
  \int_{\Sph^{n-2}} f_{2}(\omega^\p) \,
    C^{\frac{\bullet}{2}}_{0}\big(\scal{\theta}{\omega^\p}\big) \, d\omega^\p = \vol\big(\Sph^{n-2}\big)
  \neq 0 \;,
 \end{equation}
which shows that both layer potentials of $f$ will have logarithmic terms in their expansions at $Z$. On the other hand, as $C^{\frac{\bullet}{2}}_{j-2}(\scal{\theta}{\omega^\p})$ is an even resp. odd function of $\omega^\p$ if and only if $j$ is even or odd, respectively and because $g_{3+4j} = (-1)^j$ and equal to $0$ else, we see that
 \begin{equation}
  \int_{\Sph^{n-2}} g_j(\omega^\p) \,
    C^{\frac{\bullet}{2}}_{j+2}\big(\scal{\theta}{\omega^\p}\big) \, d\omega^\p = 0
 \end{equation}
for all $j \ge 2$. Hence the layer potentials of $g$ will not have logarithmic terms in their expansions.
\end{exa}

\begin{exa} Now suppose $f$ is a homogeneous polynomial on $\R^{n-1}$ of degree $-l > 1$. Clearly, $f$ is polyhomogeneous, $f \in \calA^l(Y)$, which shows that the unmodified layer potentials of $f$ are not defined. But the modified versions $\Op{\single_{k_+(l)}}$ and $\Op{\double_{k_-(l)}}$ are. The first set of conditions in \cref{eq:log.condition.mod} is vacious and the second gives a single condition:
 \[ \int_{\Sph^{n-2}} f_l(\omega^\p) C^{\frac{\bullet}{2}}_{\pm 1 - l}\big(\scal{\theta}{\omega^\p}\big)d\omega^\p = 0 \;, \]
where $f(y^\p) = (r^\p)^l f_l(\omega^\p)$. Choosing $f_l$ to have the same parity as $l$, the product $f_l \cdot C^{\frac{\bullet}{2}}_{\pm 1 - l}$ is odd in $\omega^\p$ and the condition is satisfied resulting in layer potentials that are polyhomogeneous with respect to the integer index sets $l-1$ respectively $l$.
\end{exa}

\subsection{The Boundary Layer Kernel}\label{ssec:bdy.kernel}

We now turn our attention to the boundary single layer kernels $\calN_k$. (Recall that the boundary double layer kernels $\calK_k$ all vanish identically.) Firstly, we show that they define maps between spaces of polyhomogeneous functions and then relate them to boundary values of layer potentials.

\paragraph{Polyhomogeneity} With calculations in local coordinates similar to those in \Cref{APP.ssec:poly.calc}, we can deduce that the (modified) boundary layer kernels give rise to polyhomogeneous conormal $b$-densities on $Y^2_b$. Recall that the spaces $\calI^{s,\calE}\big(Y^2_b,D(Y^2_b);\bOmega\big)$ are spaces of polyhomogeneous $b$-densities that have conormal singularities at the lifted diagonal $D(Y^2_b)$, cf.~\Cref{APP.def:phg.conormal}.

\begin{prop}\label{thm:bdy.kernel.phg} We have
 \[ \bcalN \in \calI^{-1,(-1,-1,-n)}\big(Y^2_b,D(Y^2_b);\bOmega\big) \quad\text{ and }\quad
    \bcalN_k \in \calI^{-1,(2-n-k,k-1,-n)}\big(Y^2_b,D(Y^2_b);\bOmega\big) \;, \]
where $k \in \N$ and index families are written with respect to the ordering $(\lf,\rf,\bff)$.
\end{prop}

\begin{proof} First of all, note that the density $|dy dy^\p|$ pulls back to be a polyhomogeneous $b$-density on $Y^2_b$ with index family $(1-n,1-n,2-2n)$:
 \[\begin{aligned}
  \big|dy dy^\p\big|
   = r^{1-n} (r^\p)^{1-n} \big|\bdens{r} \bdens{r^\p} d\omega d\omega^\p \big|
   &= (r^\p)^{2-2n} \big(\tfrac{r}{r^\p}\big)^{1-n} \big|\bfracdens{r}{r^\p} \bdens{r^\p} d\omega d\omega^\p\big| \\
   &= \big(\tfrac{r^\p}{r}\big)^{1-n}r^{2-2n} \big|\bfracdens{r^\p}{r} \bdens{r} d\omega d\omega^\p\big|
 \end{aligned}\]
With regard to the unmodified kernel, in local coordinates near a product neighbourhood of $\rf$, we may write
 \[ |y-y^\p|^{2-n}
     = r^{n-2} \big(\tfrac{r^\p}{r}\big)^{n-2} \big|\big(\tfrac{r^\p}{r}\big)\omega - \omega^\p\big|^{2-n} \;, \]
which is certainly polyhomogeneous with index set $n-2$ at both $\rf$ and $\bff$, with bounded and smooth coefficients as long as we stay away from the lifted diagonal. Using coordinates $r^\p$, $\tfrac{r}{r^\p}$ instead, we obtain the same result for a neighbourhood of $\lf$ which does not intersect the lifted diagonal. Near the interior of the lifted diagonal, it is clear that $|y-y^\p|^{2-n}$ has a conormal singularity of order $-\dim Y -(2-n) = -1$. Then, letting $t = 1 - \tfrac{r^\p}{r}$, $R = |\omega - \omega^\p|$ and $\Omega = R^{-1}(\omega-\omega^\p)$, the set $(r, t, \omega, R, \Omega)$ gives local coordinates valid near the boundary of the lifted diagonal, where $r$ is a boundary defining function for $\bff$ and the lifted diagonal is given by $t = R = 0$. In these coordinates,
 \[ |y-y^\p|^{2-n} = r^{n-2} (1-t)^{n-2} \big[ R^2 - 2tR \scal{\omega}{\Omega} + t^2 \big]^{\tfrac{2-n}{2}} \;, \]
and we note that this gives a polyhomogeneous expansion at $\bff$ with a single coefficient, which is homogeneous of degree $2-n$, jointly in $(t,R)$. Consequently, we have
 \[ \bcalN \in \calI^{-1,(-1,-1,-n)}\big(Y^2_b,D(Y^2_b);\bOmega\big) \;. \]
As for the modified boundary layer densities, we again note that the modification terms are smooth in the interior of $Y^2_b$ (due to the harmonicity considerations in $z$ and the cut-off in $z^\p$) and lift to be polyhomogeneous on $Y^2_b$ with index families $(-m,m+n-2,n-2)$. Using a similar argument as in the case of the full layer kernels (i.e.~using the convergent multipole expansion near $\rf$), we see that, away from the lifted diagonal, $\calN_k$ is polyhomogeneous with index family $(1-k,n+k-2,n-2)$. As the conormal singularity at the diagonal is unchanged, multiplication with the $b$-density factor completes the proof.
\end{proof}

\Cref{thm:bdy.kernel.phg} in fact tells us that $\calN$ and $\calN_k$ define $b$-pseudodifferential operators, compare \Cref{APP.def:b.pseudo} or \cite[pp.~7]{GH14}, and thus bounded operators between spaces of polyhomogeneous functions and between weighted $b$-Sobolev spaces. For this, it is convenient to identify $\calN$ and $\calN_k$ with $b$-half-densities instead of $b$-densities. This can be done by simply replacing $\big|\bdens{r}\bdens{r^\p}d\omega d\omega^\p\big|$ with its square root. This is a smooth and non-vanishing $b$-half density and in particular, we do not alter the index families. We then use the corresponding smooth and non-vanishing $b$-density $\mathrm{dvol}_b = r^{n-1} |dy|$ to define the $b$-Sobolev spaces. But observe that, after applying the operators, we identify the result first with a $b$-density and then with a function via \Cref{eq:fct.dens.ident.detail.2}.

\begin{bprop}\label{thm:bdy.mapping}\begin{enumerate}
 \item Let $E$ be an index set for $Y$ so that $\Re E > 1-k$ for $k \in \N_0$. Then $\calN$ resp. $\calN_k$ define continuous operators
   \[ \Op{\calN} : \calA^E(Y) \too \calA^{(E-1) \bcup (n-2)}(Y) \quad\text{ and }\quad
      \Op{\calN_k} : \calA^E(Y) \too \calA^{(E-1) \bcup (1-k)}(Y) \;. \]
 \item If $\alpha$, $\beta$, $m \in \R$ satisfy $\alpha > 1$, $\beta < n-2$ and $\beta-\alpha \le -1$, then $\calN$ defines a bounded operator $$r^\alpha H_b^{m-1}(Y,\mathrm{dvol}_b) \too r^\beta H_b^{m}(Y,\mathrm{dvol}_b) \;.$$ It is compact if and only if $\beta-\alpha < -1$.
 \item If $\alpha$, $\beta$, $m \in \R$ satisfy $\beta < 1-k < \alpha$ and $\beta-\alpha \le -1$, then $\calN_k$ defines a bounded operator $$r^\alpha H_b^{m-1}(Y,\mathrm{dvol}_b) \too r^\beta H_b^{m}(Y,\mathrm{dvol}_b) \;.$$ It is compact if and only if $\beta-\alpha < -1$.
\end{enumerate}\end{bprop}
But unfortunately, there is so far no analytic Fredholm theory for general elliptic elements of the full calculus. Using such an extension of the theory, it would be possible to derive detailed Fredholm and invertibility statements for $\Op{\calN}$ and $\Op{\calN_k}$.

\paragraph{Jump Formulae} Next, we prove a variant of the classical relations \Crefrange{intro:jump.1}{intro:jump.3}. We have shown that the layer potentials are smooth up to $\{x=0\}$, which is to say that they have a well-defined limit as we approach $Y$, given by the leading term in their polyhomogeneous expansions at $Y$. The subleading order terms depend on the choices of coordinates, which is why we will express them in specific coordinates $(x,r,\omega)$ and $(xr,r,\omega)$, only. In the following, $\chi$ is a boundary defining function for $Y$ and we work in a trivialisation $[0,1)_{\chi} \times Y$ of a product neighbourhood $\mathcal{V}$ of $Y$ in $X$.

\begin{thm}\label{thm:formulae} Let $f \in \calA^E(Y)$ and $k = k_\pm(E)$ be chosen correspondingly for the single or double layer potentials. Suppose $\mathcal{V} \subset X$ and coordinates $(\chi,y)$ are chosen as in the preceding paragraph. Then, there are $R_k^\pm \in \dot{C}^0\big([0,1);\calA(Y)\big)$ so that on $\mathcal{V}$,
 \begin{align}
   \Op{\single_k}f(\chi,y) &= \Op{\calN_k}f(y) + R^+_k(\chi,y) \;,
    \label{eq:jump.single} \\
   \Op{\double_k}f(\chi,y) &= \phantom{-} \tfrac{1}{2}f(y) + R_k^-(\chi,y) \;,
    \label{eq:jump.double}\\
   \partial_{\nu}\big(\Op{\single_k}f\big)(\chi,y) &= -\tfrac{1}{2}f(y) - R_k^-(\chi,y) \;,
    \label{eq:jump.single.normal}
 \end{align}
where we write $\single_0 = \single$ etc. Moreover, in coordinates $(x,r,\omega)$ near $\mathring{Y}$ and $(xr,r,\omega)$ near $\Gamma$, we have:
 \begin{align}
  \Op{\single_k}f(x,r,\omega) &= \Op{\calN_k}f(r,\omega) + x \big( \tfrac{1}{2} \, f(r,\omega) \big)
   + x\, R_k^-(x,r,\omega)
   \label{eq:jump.single.local.1} \\
  \Op{\single_k}f(xr,r,\omega) &= \Op{\calN_k}f(r,\omega) + (xr) \big( \tfrac{1}{2r} \, f(r,\omega)\big)
   + (xr)\,\big(r^{-1} R_k^-(\tfrac{xr}{r},r,\omega)
   \label{eq:jump.single.local.2}
 \end{align}
\end{thm}
The proof, given in \Crefrange{thm:jump.lemma.1}{thm:jump.lemma.4} in \Cref{APP.sec:bdy}, uses \Cref{thm:logs.exa} to explicitly calculate the leading order terms of the single and double layer potentials and uses the classical symmetry argument to prove \cref{eq:jump.single.normal}, which then in turn leads to \Crefrange{eq:jump.single.local.1}{eq:jump.single.local.2}. Note that, by definition of polyhomogeneity, \Cref{APP.def:phg}, this shows that the limits in \Cref{thm:intro.solns} are attained uniformly on compact subsets including all derivatives in directions tangent to $Y$.

\clearpage\appendix
\section{Calculations and Proofs}\label{APP.sec:calcs}

\subsection{Regarding Polyhomogeneity}\label{APP.ssec:poly.calc}

Having fixed adapted coordinates in $Y$ and $X$ in \cref{SK.eq:corner.coordinates}, we will use the corresponding (projective) adapted coordinates in $\calPD$. For instance, coordinates near the interior of $\lf(Y) \cap \df$ (that is, away from the face $\bff$), are given by $\big(\tfrac{x}{s}, s, \eta, y\big)$, where $s = \big|y-y^\p\big|$ and $\eta = \tfrac{y-y^\p}{s}$, and coordinates near $\rf \cap \lf(Y) \cap \bff$ are given by $\big(xr, r, \tfrac{r^\p}{r}, \omega, \omega^\p\big)$.

\begin{lemma}\label{APP.thm:b.exponents} Let $\pi_l$, $\pi_r$ be the compositions of the blow-down map $\calPD \xrightarrow{\;\beta_\calPD\;} X \times Y$ with the projections $\pr_l$, $\pr_r$ onto the left respectively right factor of $X \times Y$. $\pi_l$ and $\pi_r$ are $b$-fibrations and their exponent matrices are given by
 \begin{equation}\label{eq:b.boundary.exponents}
  e_{\pi_l} = \begin{BMAT}[0pt]{cc}{cc}
                 \begin{BMAT}(1.4em,1.4em){cc}{c} \;\;\hspace{-7pt}\scr{Y} & \scr{Z} \end{BMAT} & \\
                 \left(\begin{BMAT}(b,1.4em,1.4em){cc}{ccccc}
                   1 & 0  \\  0 & 1  \\ 0 & 0 \\   0 & 1  \\  1 & 0
                 \end{BMAT}\right)
                 & \begin{BMAT}(b,1.45em,1.45em){c}{bbbbb}
                     \rscr{\lf(Y)}{1} \\ \rscr{\lf(Z)}{2} \\ \rscr{\rf(\Gamma)}{1} \\ \rscr{\bff}{3} \\ \rscr{\df}{3}
                   \end{BMAT}
                \end{BMAT}
  \qquad\text{ and }\qquad
  e_{\pi_r} = \begin{BMAT}[0pt]{cc}{cc}
                 \begin{BMAT}(b,1.4em,1.4em){c}{c} \scr{\Gamma} \end{BMAT} & \\
                 \left(\begin{BMAT}(b,1.4em,1.4em){c}{ccccc}
                   0 \\ 0 \\ 1 \\ 1 \\ 0
                 \end{BMAT}\right)
                 & \begin{BMAT}(b,1.45em,1.45em){c}{bbbbb}
                     \rscr{\lf(Y)}{1} \\ \rscr{\lf(Z)}{2} \\ \rscr{\rf(\Gamma)}{1} \\ \rscr{\bff}{3} \\ \rscr{\df}{3}
                   \end{BMAT}
                \end{BMAT}
  \qquad.
 \end{equation}
\end{lemma}

\begin{proof} Since we have homogeneously blown up the pull-backs of distinct boundary faces of $X \times Y$, namely $Z \times \Gamma$ and the lift of the diagonal, the exponent matrix $e_{\beta_\calPD}$ has entries in $\{0,1\}$. In particular, $e_{\beta_\calPD}(G,H) = 1$ if and only if $\beta_\calPD(G) \subset H$. A projection onto one of the factors of $X \times Y$ is a $b$-map with boundary exponents in $\{0,1\}$ as well and an exponent is $1$ if and only if the first face is projected onto the second. Therefore, the blow-down projections $\pi_\bullet$ are $b$-maps with boundary exponents given by
 \[ e_{\pi_\bullet}(G,H) = \begin{cases} 1 &\text{if $\pi_\bullet(G) \subset H$,} \\
                                                    0 &\text{else,} \end{cases} \]
for $\bullet=l,r$ and where $G$ is a boundary hypersurface of $\calPD$ and $H$ one of either $X$ or $Y$. This easily gives \cref{eq:b.boundary.exponents} and we are left with showing that $\pi_l$ and $\pi_r$ are $b$-fibrations.

Now note that the projections $\pr_l$, $\pr_r$ are not only $b$-maps but even $b$-fibrations (in fact, they are genuine fibrations), and that the composition of a $b$-fibration with the blow-down map for a blow-up of a boundary $p$-submanifold is again a $b$-fibration, see the proof of Proposition 6 in \cite{MM98} and compare Lemma 6 of \cite{GH09}. Since in the construction of $\calPD$ we have blown up boundary $p$-submanifolds only, this implies that the blow-down projections $\pi_l$ and $\pi_r$ are in fact $b$-fibrations.
\end{proof}

\begin{lemma}\label{APP.thm:x.phg} The function $\scal{\nu(z^\p)}{z-z^\p}=x$ lifts to be polyhomogeneous on $\calPD$ with respect to the index family $(1,-1,0,-1,1)$.
\end{lemma}

\begin{proof} Please observe that, in local coordinates near the interior of $Y \subset X$, the interior of $Z$ and near the corner $\Gamma$, respectively, we have
 \[ x = x^1 = \rho^{-1}\sqrt{1-\rho^2|y|^2} = (xr)^1 r^{-1} \;. \]
Consequently, $x$ is polyhomogeneous on $X$ with index family $(1,-1)$ for $\big(Y,Z\big)$. The Pull-Back Theorem in conjunction with the exponent matrix \cref{eq:b.boundary.exponents} of $\pi_l$ then shows that $\beta_\calPD^*(x)$ is polyhomogeneous with index family $(1,-1,0,-1,1)$ on $\calPD$.
\end{proof}

\begin{lemma}\label{APP.thm:abs.phg} For any $\lambda \in \R$, the pull-back of the functions $|z-z^\p|^\lambda$ to $\calPD$ is polyhomogeneous with index family $(0,-\lambda,-\lambda,-\lambda,\lambda)$.
\end{lemma}

\begin{proof} In the $b$-double space of $\overline{\R^n}$ yet away from the diagonal, the function $|z-z^\p|^\lambda$ lifts to be polyhomogeneous with index family $(-\lambda,-\lambda,-\lambda)$ at the faces $(\lf,\rf,\bff)$. If we then blow up the diagonal inside the $b$-double space, it is clear that $|z-z^\p|^\lambda$ has index set $\lambda$ at this newly created face. Now the embedding $\calPD \hookrightarrow M$ of $\calPD$ into the $b$-double space with resolved diagonal is a $b$-map and $|z-z^\p|^\lambda$ is clearly smooth and non-vanishing across the lifts of $\{z=0\}$ and $\{z^\p=0\}$. Determining the boundary exponents of the embedding and using the Pull-Back Theorem, we see that $|z-z^\p|^\lambda$ lifts to be polyhomogeneous with index family $(0, -\lambda, -\lambda, -\lambda, \lambda)$ at $(\lf(Y),\lf(Z),\rf,\bff,\df)$ in $\calPD$.
\end{proof}

\begin{lemma}\label{APP.thm:mod.phg} For any $m \in \N_0$, the functions $\psi(z^\p) \tfrac{|z|^m}{|z^\p|^{m+n-2}} C^{\frac{n-2}{2}}_m(\Theta)$ and $\psi(z^\p) x \tfrac{|z|^m}{|z^\p|^{m+n}} C^{\frac{n}{2}}_m(\Theta)$ lift to be polyhomogeneous on $\calPD$. In particular,
 \begin{align*}
  \psi(z^\p) \tfrac{|z|^m}{|z^\p|^{m+n-2}} C^{\frac{n-2}{2}}_m(\Theta)
   &\in \calA^{(0,-m,m+n-2,n-2,0)}\big(\calPD\big) \;, \\
  \psi(z^\p) x \tfrac{|z|^m}{|z^\p|^{m+n}} C^{\frac{n}{2}}_m(\Theta)
   &\in \calA^{(1,-1-m,m+n,n-1,1)}\big(\calPD\big) \;.
 \end{align*}
\end{lemma}

\begin{proof} Let us consider the term $A = \psi(z^\p) |z|^m |z^\p|^{2-m-n} C^{\frac{n-2}{2}}_m(\Theta)$ first. Following \Cref{thm:hom.harm.decomp}, $A$ is harmonic and hence smooth in $z$. Since $\psi(z^\p)$ is supported away from $z^\p = 0$ and because $C^{\frac{n-2}{2}}_m$ is a polynomial of $\Theta = \tfrac{<z,z^\p>}{|z| |z^\p|}$, it is smooth in $z^\p$ as well. Moreover, using the calculation
 \[ \Theta = \tfrac{<z,z^\p>}{|z||z^\p|} = \scal{\rho}{\omega^\p} = \scal{\omega}{\omega^\p} \big(1 + (xr)^2\big)^{-\tfrac{1}{2}} \]
we see that $\Theta$ lifts to be polyhomogeneous with index sets $0$ at every face of $\calPD$. As $C^{\frac{n-2}{2}}_m(\Theta)$ is a polynomial of $\Theta$, the only non-zero index sets for $A$ on $\calPD$ will arise from the factors $\tfrac{|z|^{m}}{|z^\p|^{m+n-2}}$. Using local coordinates in $X \times Y$ near $Z \times Y$, we get
 \begin{equation}\label{eq:mod.phg.1}\begin{aligned}
  \frac{|z|^m}{|z^\p|^{m+n-2}}
   &= \rho^{-m} |z^\p|^{2-m-n} = \rho^{-m} (r^\p)^{m+n-2} \\
   &= r^{-m} |z^\p|^{2-m-n} \big(1 + (xr)^2\big)^{-\tfrac{1}{2}}
      = r^{-m} (r^\p)^{m+n-2} \big(1 + (xr)^2\big)^{-\tfrac{1}{2}} \;,
 \end{aligned}\end{equation}
and this shows that $\tfrac{|z|^m}{|z^\p|^{m+n-2}}$ is polyhomogeneous on $X \times Y$ with index sets $(0,-m,m+n-2)$ at $(Y \times Y, Z \times Y, X \times \Gamma)$. The Pull-Back Theorem then gives the first of our claims.

As for the term $B=\psi(z^\p) \psi(z^\p) x |z|^m|z^\p|^{-m-n} C^{\frac{n}{2}}_m(\Theta)$, combining the same arguments as before with \Cref{APP.thm:x.phg}, we see that $B$ lifts to be polyhomogeneous on $\calPD$ and that the only non-zero index sets arise from the factors $x \frac{|z|^{m}}{|z^\p|^{m+n}}$. Using \cref{eq:mod.phg.1}, \Cref{APP.thm:x.phg} and the Pull-Back Theorem, we arrive at the second claim.
\end{proof}

\begin{lemma}\label{APP.thm:dens.phg} We have $\mu \in \calA^{(0,1-n)}(X;\bOmega)$, $|dz^\p| \in \calA^{1-n}(Y;\bOmega)$ and
 \[ \pi_l^\ast(\mu) \, \pi_r^\ast(|dz^\p|) \in \calA^{(0,1-n,1-n,2-2n,n-1)}\big(\calPD;\bOmega\big) \;. \]
\end{lemma}

\begin{proof} It is clear that $|dz|$ and $|dz^\p|$ are smooth densities on the interior of $X$ respectively $Y$. Near the boundaries, direct calculations give
 \begin{equation}\label{eq:dens.factors.bdy}\begin{aligned}
  \big|dz\big| &= x \big| \bdens{x} dy \big| = \rho^{-n} \big|\bdens{\rho}d\vartheta\big| = (xr) r^{-n} \big|\bdens{x}\bdens{r}d\omega\big| \\
  \big|dz^\p\big| &= (r^\p)^{1-n}\big|\bdens{r^\p}d\omega^\p\big|
 \end{aligned}\end{equation}
and this shows that $|dz| \in \calA^{(1,-n)}\big(X;\bOmega\big)$ and $|dz^\p| \in \calA^{1-n}\big(Y;\bOmega\big)$. Multiplication with $x^{-1}$ then shows that $\mu \in \calA^{(0,1-n)}\big(X;\bOmega\big)$. Using respective coordinates near codimension 3 corners not involving $\df$ in $\calPD$, we see that, for instance
 \begin{equation}\label{eq:dens.total.bdy}
  \mu \big|dz^\p\big|
    = \big(\tfrac{r^\p}{r}\big)^{1-n} r^{2-2n}
        \big|\bdens{x}\bdens{r}\bfracdens{r^\p}{r}d\omega d\omega^\p\big|
 \end{equation}
and we obtain index families $\big(0,1-n,1-n,2-2n,\bullet\big)$. Then, using coordinates near the interior of $\lf(Y) \cap \df$, e.g., $s = |y-y^\p|$, $\eta = s^{-1}(y-y^\p)$ and $\tfrac{x}{s}$, we get
 \[ \mu \big|dz^\p\big| = s^{n-1} \big|\bfracdens{x}{s} \tfrac{ds}{s} d\eta dy\big| \;. \]
Similar representations hold in coordinate systems valid near the other corners, in particular near $\bff \cap \df$, and we conclude that $\mu|dz^\p|$ has index family $(0,1-n,1-n,2-2n,n-1)$.
\end{proof}

\subsection{Regarding Logarithmic Terms}\label{APP.sec:logs}

We start by giving a proof of the central \Cref{thm:log.lemma}. This result will allow us to show that the layer potentials are smooth up to $Y \subset X$ and to obtain conditions for the appearance of logarithms in their expansions at $Z$.

\begin{proof}[of \Cref{thm:log.lemma}] As $F$ is a composition of the projection off of $\Sph^m$, which is a genuine fibration, with the map $(a,b) \mapsto ab$, which is shown to be a $b$-fibration in \cite{gohar}, for instance, we see that $F$ is a fibration over the interior of any face of its range. Moreover, any boundary hypersurface $G$ of $[0,1)^2 \times N \times \Sph^m$ is of either two forms, containing a factor $\Gamma_x \in \calM_1([0,1)^2)$ or $\Gamma_N \in \calM_1(N)$. As we have $F(\Gamma_x) \subset \{0\}\times N$ and $F(\Gamma_N) \subset [0,1) \times \Gamma_N$, we see that the codimension of boundary hypersurfaces is not increased, which shows that $F$ is a $b$-fibration. Moreover, it shows that $\nul_{e_F} = \emptyset$ which is to say that the integrability condition for the Push-Forward Theorem, \cref{APP.eq:pft.int.cond}, is vacious and that $F_* u$ is well-defined and polyhomogeneous at $\{0\} \times N$ with index family $E_1 \bcup E_2$.

Now, as $u$ is compactly supported, we may write its push-forward as
 \begin{equation}
  F_*u(t,y) = \left(
      \int_0^\infty \int_{\Sph^m} \widetilde{u}(x_1,\tfrac{t}{x_1},y,\theta) d\theta \bdens{x_1} \right)
     \bdens{y}\bdens{t} \;,
 \end{equation}
for a suitable function $\widetilde{u}$, compare \cref{eq:pft.sal.ex}. As the index sets $E_1$, $E_2$ are integer, we have
 \[ \widetilde{u} \sim \sum_{i,j} x_1^i x_2^j \widetilde{u}_{ij}(y,\theta) \]
near $\{0\}^2 \times N \times \Sph^m$ for functions $\widetilde{u}_{ij} \in \calA^{\mathcal{E}_0}(N \times \Sph^m)$, where $\mathcal{E}_0 = \mathcal{E}_{\{0\}^2\times N \times \Sph^m}$. Now, by assumption we have $u(x_1,x_2,y,\theta) = -u(-x_1,-x_2,y,-\theta)$ which implies
 \[ x_1^i x_2^j \widetilde{u}_{ij}(y,\theta) = (-1)^{i+j+1}x_1^i x_2^j \widetilde{u}_{ij}(y,-\theta) \]
and in particular $\widetilde{u}_{ii}(y,\theta) = - \widetilde{u}_{ii}(y,-\theta)$. Since we can change the order of integration and asymptotic summation because of the integral's uniform convergence (the respective coefficients of $\widetilde{u}$ are compactly supported and smooth), we see that the coefficients $b_i$ from \Cref{thm:logs.exa} are given by
 \[ b_i = \int_{\Sph^m} \widetilde{u}_{ii}(y,\theta) = \int_{\Sph_+^m} \widetilde{u}_{ii}(y,\theta) + \widetilde{u}_{ii}(y,-\theta) = 0 \;. \]
Thus, $F_*u$ is in fact polyhomogeneous with index set $E_1 \cup E_2$ at $\{0\} \times N$.
\end{proof}

\paragraph{Smoothness up to $\mathbf{Y}$} \Cref{thm:logs.exa,thm:log.lemma} are local models for the push-forward along $\pi_l$ and as polyhomogeneity is a local property, we can apply these by using a partition of unity subordinate to a suitable cover. For ease of reference, we will use the partition of unity only implicitly.

To start with, we apply \Cref{thm:log.lemma} to the unmodified layer densities and show that the corresponding layer potential operators give functions that are smooth up to the boundary hypersurface $Y \subset X$. Observe that in order to apply \Cref{thm:log.lemma}, we need not only a $b$-fibration of a specific form but also a $b$-density that is polyhomogeneous with integer index sets at the faces in question.

\begin{prop}\label{thm:smooth.x} Given any index set $E$ and $f \in \calA^E(Y)$, the unmodified layer potentials of $f$, $\Op{\single}f$ and $\Op{\double}f$, are smooth up to the face $Y \subset X$.
\end{prop}

\begin{proof} We need to consider integrals over fibres of $\pi_l$ as the base point approaches $Y \subset X$. Looking at the definition of push-forwards of index sets, \cref{APP.eq:pft.index}, we see that the only faces involved are $\lf(Y)$ and $\df$ and that the only obstruction to smoothness is the possible appearance of logarithmic terms, arising from the fibres of $\pi_{l}$ being pushed into the corner $\lf(Y) \cap \df$. First note that the Pull-Back Theorem shows that $\pi_{r}^\ast (f)$ is polyhomogeneous with index families $0$ at both $\lf(Y)$ and $\df$ and hence polyhomogeneous with respect to integer index sets for these faces for any choice of index set $E$. Then, using \cref{eq:kernels.phg}, the Push-Forward Theorem and \cref{eq:fct.dens.ident.detail.1}, we see that $\Op{\single}f$ and $\Op{\double}f$ are polyhomogeneous at $Y \subset X$ with index sets $0 \bcup 1$. We will now use \Cref{thm:log.lemma} to substitute a normal union for the extended union.

We start by considering fibres of $\pi_{l}$ over points in $X$ approaching the interior of $Y$. Near this corner, let us choose suitable local coordinates
 \begin{equation}
  \Big( \tfrac{x}{s}, s, \eta, y \Big)
   = \Big( \frac{x}{|z-z^\p|}, |z-z^\p|, \frac{z-z^\p}{|z-z^\p|}, y \Big) \;,
 \end{equation}
noting that $\tfrac{x}{s}$ is a boundary defining function for $\lf(Y)$, $(s, \eta)$ are polar coordinates around $\df$ and $y$ are coordinates in $\lf(Y) \cap \df$. Then,
 \begin{equation}
  \pi_{l}\big( \tfrac{x}{s}, s, \eta, y \big) = \big( \tfrac{x}{s}s,y\big)
 \end{equation}
which is to say that the push-forward amounts to integration over the unit normal sphere bundle to $\df$, $\eta \in \Sph^{n-1}$, and push-forward along the map $(\tfrac{x}{s},s) \mapstoo \tfrac{x}{s}s$. In these coordinates and on a fibre over a point near $\mathring{Y}$ (and implicitly using cut-off functions), the $b$-densities under consideration are given by
 \begin{equation}\label{eq:smoothness.layer.pots.local}\begin{aligned}
  \bsingle \, \pi_{r}^\ast (f) &= s^1 f(y-s\eta) \widetilde{\mu} \;, \\
  \bdouble \, \pi_{r}^\ast (f) &= \big(\tfrac{x}{s}\big) f(y-s\eta) \widetilde{\mu} \;,
 \end{aligned}\end{equation}
where $\widetilde{\mu}$ is a smooth and non-vanishing $b$-density. But then, we are in the situation of \Cref{thm:log.lemma}, with $N = \lf(Y) \cap \df$, $m = n-1$ and for $F = \pi_{l}$. As \cref{eq:smoothness.layer.pots.local} gives \cref{eq:log.cond.1.change}, both lines in \cref{eq:smoothness.layer.pots.local} satisfy the prerequisites of \Cref{thm:log.lemma}, on fibres over points near the interior of $Y$.

Analogously, we may choose adapted local coordinates near the corner $\lf(Y) \cap \df \cap \bff$ (cf.~the proof of \Cref{thm:smooth.x.mod}) and obtain a representation similar to \cref{eq:smoothness.layer.pots.local}. Then, using \Cref{thm:log.lemma} again, we conclude that the layer potentials of $f$ (identified with functions, as usual) are polyhomogeneous at $Y$ with index sets $0 \cup 1 = 0$, i.e., smooth up to the face $Y \subset X$.
\end{proof}

The modified layer densities differ from the unmodified ones by a finite number of modification terms and hence it is sufficient to show that these modification terms give rise to functions that are smooth up to the face $Y$. We will not use \Cref{thm:log.lemma} in this case but show directly that there are no diagonal elements.

\begin{prop}\label{thm:smooth.x.mod} Given any index set $E$ and $f \in \calA^E(Y)$, the modified layer potentials of $f$, $\Op{\single_k}f$ and $\Op{\double_k}f$, are smooth up to the face $Y \subset X$.
\end{prop}

\begin{proof} As before, we need only consider those parts of fibres of $\pi_{l}$ which are close to $\lf(Y) \cap \df$. Thanks to the cut-off function $\psi$, it is in this case sufficient to consider the situation near $\partial(\lf(Y)\cap\df)$, as the type of local product coordinates we will be using there extend to local product coordinates for a neighbourhood of $\supp \psi \cap \lf(Y) \cap \df$. (This is of course due to standard coordinates in $\R^n$ being global coordinates.)

Let us choose coordinates in a neighbourhood of $\partial\big(\beta_\calP^\ast(D)\big)$, the boundary of the diagonal lifted to $\calP$, for instance $(xr, r, t, \omega, \kappa)$, where $t = 1-\tfrac{r^\p}{r}$, $\kappa = \omega^\p-\omega$. Then,
 \begin{equation}\label{eq:smoothness.mods.local} \begin{aligned}
   &\frac{|z|^m}{|z^\p|^{m+\lambda}} f(y^\p) C_m^{\frac{\lambda}{2}}(\Theta) \\
     &\qquad= r^\lambda (1-t)^{m+\lambda} \Big[ s^2 (xr)^2 + 1 \Big]^m \\
     &\qquad\quad\, \cdot\, f\big( r^{-1} \tfrac{1}{1-t} (\omega + \kappa) \big) \;\cdot\;
            C_m^{\frac{\lambda}{2}}\Big( r\big(1+(xr)\big)^{-\tfrac{1}{2}} \big(1+\scal{\omega}{\kappa}\big)\Big)
 \end{aligned}\end{equation}
where $s^2 = t^2 + |\kappa|^2$ is a defining function for the lifted diagonal $\beta_\calP^*(D) = \{ t=\kappa=0\}$, and the density factor takes the form
 \[ r^{2-2n} t (1-t)^{2-n} \left| \bdens{xr} \bdens{r} \bdens{t} d\omega d\kappa \right| \;. \]
The right hand side of \cref{eq:smoothness.mods.local} are polyhomogeneous functions, and its coefficients (as $xr \to 0$) are smooth across the lifted diagonal. Moreover, when multiplied with the $b$-density factor, these coefficients vanish to at least first order in $t$.

Now we lift things to a neighbourhood of $\lf(Y) \cap \df \cap \bff$ in $\calP_D$. We use local product coordinates $(\tfrac{xr}{s}, r, s, \tfrac{t}{s}, \tfrac{\kappa}{s},\omega)$ and assume that $u$ is a polyhomogeneous $b$-density on $\calP$, supported near $\partial\big(\beta_\calP^\ast(D)\big)$, so that
 \[ u \sim \sum_{j \ge k} \sum_{i \ge l} t (xr)^j r^i \, u_{ij} \;. \]
Then, near $\lf(Y) \cap \df \cap \bff$ we have
 \[ \beta_\calPD^*(u) \sim \sum_{j \ge k} \sum_{i \ge l} s^{j+1} \big(\tfrac{xr}{s}\big)^j \tfrac{t}{s} r^i \widetilde{u}_{ij} \;. \]
Because $\tfrac{xr}{s}$ and $s$ are boundary defining functions for $\lf(Y)$ and $\df$ respectively, this shows that there are no diagonal elements, by which we mean terms corresponding to $(\tfrac{xr}{s})^j s^j$ for $j \in \Z$. Applying this argument to the $b$-densities corresponding to the modification terms, $\big( \bsingle - \bsingle_k \big) \pi_{r}^\ast(f)$ and $\big( \bdouble - \bdouble_k \big) \pi_{r}^\ast (f)$, we see that its push-forward along $\pi_{l}$ does not give rise to any additional logarithmic terms, by \Cref{thm:logs.exa}. As there have been none in the first place and by using \Cref{thm:smooth.x}, we see that the modified layer potentials $\Op{\single_k}f$ and $\Op{\double_k}f$ are smooth up to $Y \subset X$ for any $k \in \N$ and any $f \in \calA^E(Y)$.
\end{proof}

\paragraph{Logarithmic Behaviour at $\mathbf{Z}$} We now turn to proving the propositions leading to \Cref{ex:logs}, \Cref{thm:logs.condition,thm:logs.condition.mod}.

\begin{proof}[of \Cref{thm:logs.condition}]\label{pf:exa.1} Following the same arguments as in the proofs of \Cref{thm:smooth.x,thm:smooth.x.mod}, it is sufficient to consider fibres of $\pi_{l}$ which intersect a neighbourhood of ${\lf(Z)} \cap \bff$ and in fact this intersection only. In order to be brief, we restrict our attention to a neighbourhood of $\partial ({\lf(Z)} \cap \bff) = {\lf(Z)} \cap \bff \cap \lf(Y)$, the calculations for the interior of ${\lf(Z)} \cap \bff$ are analogous. There, we can use coordinates $(r^\p, \tfrac{r}{r^\p}, xr, \omega, \omega^\p)$, where the first three are boundary defining functions: $r^\p$ for $\bff$, $\tfrac{r}{r^\p}$ for $\lf(Z)$ and $xr$ for $\lf(Y)$. Then,
 \begin{equation}
  \pi_{l}\big(r^\p, \tfrac{r}{r^\p}, xr,\omega,\omega^\p\big)
   = \big( r^\p \tfrac{r}{r^\p}, xr, \omega \big)
   = \big( r, xr, \omega \big) \;,
 \end{equation}
where we note that $r$ is a boundary defining function for $Z$ and $xr$ one for $Y$. Thus, apart from integrating over $\omega^\p \in \Sph^{n-2}$, where are in the situation of \Cref{thm:logs.exa}.

Near ${\lf(Z)} \cap \bff \cap \lf(Y)$, we have $|z| = (\tfrac{r}{r^\p})^{-1}(r^\p)^{-1} ( (xr)^2 + 1 )^{1/2}$ and $|z^\p| = (r^\p)^{-1}$. Consequently, on a sufficiently small neighbourhood, we have $(\tfrac{r}{r^\p}) < (1+(xr)^2)^{1/2}$ and equivalently $|z^\p| < |z|$. But this is to say that the \emph{other} multipole expansion
 \begin{equation}
  \big|z-z^\p\big|^{-\lambda} = \sum_{m=0}^\infty \frac{|z|^m}{|z^\p|^{m+\lambda}} C^{\frac{\lambda}{2}}_m(\Theta)
 \end{equation}
converges uniformly, where
 \[ \Theta = \tfrac{<z,z^\p>}{|z| |z^\p|} = \scal{\theta}{\omega^\p} = \big(1+(xr)^2\big)^{-\tfrac{1}{2}}\scal{\omega}{\omega^\p} \]
is independent of $r^\p$ and $\tfrac{r}{r^\p}$. The pull-backs of the single summands are of the form
 \begin{equation}
  \beta_\calPD^\ast\left(\frac{|z^\p|^m}{|z|^{m+\lambda}} C^{\frac{\lambda}{2}}_m(\Theta)\right)
   = (\tfrac{r}{r^\p})^{m+\lambda} (r^\p)^{s} \big(1+(xr)^2\big)^{-\tfrac{m+\lambda}{2}}
     C^{\frac{\lambda}{2}}_m(\Theta) \;.
 \end{equation}
Multiplying this by the terms $(r^\p)^j f_j(\omega^\p)$, $j \ge l \ge \alpha_\pm(0)$, from the data and by the density factor $\beta_\calPD^\ast\big( \pi_l^\ast(\mu) \pi_r^\ast(|dz^\p|)\big)$ yields
 \begin{equation}
  (\tfrac{r}{r^\p})^{m+\lambda+1-n} (r^\p)^{\lambda + j + 2 - 2n} f_j(\omega^\p)
    C^{\frac{\lambda}{2}}_m(\Theta) \big(1+(xr)^2\big)^{-\tfrac{m+\lambda}{2}} \widetilde{\mu} \;,
 \end{equation}
where $\widetilde{\mu}$ is a smooth and non-vanishing $b$-density. In the case of the double layer potential we need to multiply this by another factor, namely $\scal{\nu(z^\p)}{z-z^\p} = x = (\tfrac{r}{r^\p})^{-1} (r^\p)^{-1} (xr)$, for which it is important to note that the exponents of $(\tfrac{r}{r^\p})$ and $(r^\p)$ are identical. Therefore, the diagonal elements with respect to $(\tfrac{r}{r^\p})$ and $(r^\p)$ are given by choosing $m = j + 1 - n$ and are of the form
 \begin{equation}
  \Bigg\{(\tfrac{r}{r^\p})^{-1} (r^\p)^{-1} (xr) \, \cdot \Bigg\} \,
  (\tfrac{r}{r^\p})^{\lambda+j+2-2n} (r^\p)^{\lambda+j+2-2n}
  f_j(\omega^\p) C^{\frac{\lambda}{2}}_{j+1-n}(\Theta) \big(1+(xr)^2\big)^{-\tfrac{\lambda+j+1-n}{2}} \widetilde{\mu} \;,
 \end{equation}
in which the term in curly braces appears for the double layer potentials only. As $j \ge \alpha_\pm(0)$ and since $m \ge 0$, these exist for any $j \ge \max\{n-1,\alpha_\pm(0)\} = n-1$. By \Cref{thm:logs.exa}, the corresponding coefficient in the pushed-forward density vanishes identically if and only if
 \begin{equation}
  \int_{\Sph^{n-2}} f_j(\omega^\p) C^{\frac{\lambda}{2}}_{j+1-\lambda}(\Theta) d\omega^\p
  = 0 \qquad\text{for $j \ge n-1$.}
 \end{equation}
Repeating the same calculation in coordinates valid near the interior of $\lf(Z) \cap \bff$, we arrive at the claim.
\end{proof}

\begin{proof}[of \Cref{thm:logs.condition.mod}] To begin with, we will consider the modification terms only and again consider them on the intersection of fibres of $\pi_{l}$ with a neighbourhood of ${\lf(Z)} \cap \bff \cap \lf(Y)$, only. In local product coordinates near this corner, the modification terms are
 \begin{equation}
  \frac{|z|^m}{|z^\p|^{m+\lambda}} C^{\frac{\lambda}{2}}_m(\Theta)
   = (\tfrac{r}{r^\p})^{-m} (r^\p)^{\lambda} (1+(xr)^2)^{\tfrac{m}{2}} C^{\frac{\lambda}{2}}_m(\Theta)
 \end{equation}
with $0 \le m \le k-1$ and an additional factor $(\tfrac{r}{r^\p})^{-1} (r^\p)^{-1} (xr)$ for the double layer potentials. This yields terms of the form
 \begin{equation}
  \Bigg\{ \, (\tfrac{r}{r^\p})^{-1} (r^\p)^{-1} (xr) \cdot \Bigg\} \,
  (\tfrac{r}{r^\p})^{1-m-n} (r^\p)^{\lambda+j+2-2n} f_j(\omega^\p)
  C^{\frac{\lambda}{2}}_{m}(\Theta) (1+(xr)^2)^{\tfrac{m}{2}} \widetilde{\mu} \;,
 \end{equation}
where $\widetilde{\mu}$ is again a smooth and non-vanishing $b$-density. Observe that these terms do not \qt{interfere} with the terms from the unmodified layer potentials, as a comparison of exponents for $\tfrac{r}{r^\p}$ gives (if $m^\p$ denotes the $m$ from the proof of \Cref{thm:logs.condition}) $1-m-n = m^\p + \lambda +1-n$ and so $m^\p = -\lambda - m < 0$, which cannot be. The diagonal elements with respect to $r^\p$ and $(\tfrac{r}{r^\p})$ from the modification terms are obtained exactly for $m = n-1-\lambda-j$ or equivalently $j = \alpha_\pm(k) + k - m$. Thus, they exist for any $\alpha_\pm(k) + 1 \le j \le \alpha_\pm(k) + k$ respectively $\pm 1 + 1 - k \le j \le \pm 1$. Then, as before, we arrive at the claim.
\end{proof}

\subsection{Regarding Jump Relations}\label{APP.sec:bdy}

In this section of the appendix we turn to proving the jump formulae \crefrange{eq:jump.single}{eq:jump.single.normal}. For clarity, the proof is divided into four lemmata.

\begin{lemma}\label{thm:jump.lemma.1} The leading order term of $\Op{\single}f$ at $Y$ is given by $\Op{\calN}f$. More precisely, given a product neighbourhood $[0,1) \times Y$ of $Y \subset X$, there is a function $R \in \dot{C}^0\big([0,1);\calA(Y)\big)$ so that, on this product neighbourhood,
 \[ \Op{\single}f (\chi,y) = \Op{\calN}f(y) + R(\chi,y) \;. \]
\end{lemma}

\begin{proof} We already know that $\bsingle$ is polyhomogeneous with index sets $0$ and $1$ at $\lf(Y)$ and $\df$. Thus, the Push-Forward Theorem and \Cref{thm:smooth.x} show that $\Op{\single}f(\chi,y)$ has a well-defined limit as $\chi$ tends to $0$, but we do not know the actual limit. To obtain this limit, we consider the integral which defines the push-forward more closely by splitting the fibre along which we integrate into more manageable parts: Choose a trivialisation of a product neighbourhood $\mathcal{V} \cong [0,1) \times Y$ of $Y$ in $X$ and denote coordinates in $[0,1) \times Y$ by $(\chi,y)$. ($\chi$ is a global boundary defining function for $Y$, in contrast to $x$ or $xr$ which are local.) Note that this also fixes a product neighbourhood of $Y^2$ in $X \times Y$ with coordinates $(\chi,y,y^\p)$. Lift these coordinates to $\calP_D$ and let $\mathcal{U}_0$ be a product neighbourhood of $\lf(Y) \cap \df$, $\mathcal{U}_1$ be one of $\df \setminus \mathcal{U}_0$ and $\mathcal{U}_2$ be one of $\lf(Y) \setminus \mathcal{U}_0$, all in $\calP_D$. Then consider
 \[ A_0 = \pi_{l}^{-1}(\chi,y) \cap \mathcal{U}_0 \quad\text{, }\quad
    A_1 = \pi_{l}^{-1}(\chi,y) \cap \mathcal{U}_1 \quad\text{ and }\quad
    A_2 = \pi_{l}^{-1}(\chi,y) \setminus \Big(A_0 \cup A_1\Big) \;. \]
As long as $\chi$ is sufficiently small, these sets cover the fibre $\pi_{l}^{-1}(\chi,y)$. To begin with, we focus on the case $y \in \mathring{Y}$, that is we stay away from the face $\bff$.

Over $A_2$, the variables with respect to which we integrate, $y$, and take the limit, $\chi \to 0$, are independent. Or, put differently, we have $e_{\pi_{l}}(\rf_\calPD,Y) = 0$. The standard theorems on interchanging integration and taking limits then show that
 \[ \Big(\pi_l\big|_{A_2}\Big)_*\left( \bsingle \,\cdot\, \pi_{r}^*(f) \right)(\chi,y)
     \too \Big(\pi_l^\partial\big|_{B_2}\Big)_* \left( \bcalN \,\cdot\, (\pi_r^\partial)^*(f) \right) (y)
  \quad\text{as $\chi \to 0$,} \]
which is due to uniform convergence $\big(\bsingle \, \pi_r^*(f)\big)(\chi,y,y^\p) \too \big(\bcalN \, \pi_r^*(f)\big)(y,y^\p)$ as $\chi \to 0$ and where $B_2 = (\pi_l^\partial)^{-1}(y) \cap \mathcal{U}_2$, in which we deliberately confuse $\mathcal{U}_2$ with the corresponding subset of $Y^2_b$.

Over $A_1$, let $s = |y-y^\p|$, $\eta = \tfrac{y - y^\p}{s}$ and consider local coordinates $\big(\chi,\tfrac{s}{\chi},\eta,y\big)$. These are valid on $\mathcal{U}_1$ and we can again interchange the order of integration (with respect to $(\tfrac{s}{\chi},\eta)$) and taking the limit (as $\chi \too 0$). This time, since $\bsingle \pi_r^*(f)$ has index set $1$ at $\df$, we obtain
 \[ \Big(\pi_l\big|_{A_1}\Big)_* \left( \bsingle \,\cdot\, \pi_{r}^*(f) \right)(\chi,y)
     \too 0 \;. \]

Regarding $A_0$, first of all we refer to our considerations in \Cref{APP.sec:logs} and in particular to \Cref{thm:logs.exa}, again. These show that the corner $\df \cap \lf(Y)$ does not contribute any logarithmic terms and that the leading coefficient of
 \begin{equation}\label{pot.concl.eq:lead.coeff.relations}
  \Big(\pi_l\big|_{A_0}\big)_* \left( \bsingle \,\cdot\, \pi_{r}^*(f) \right) (\chi,y)
 \end{equation}
with respect to $\chi$ is given solely by integrating over $A_0^\p = \pi_{l}^{-1}(0,y) \cap \lf(Y) \cap \mathcal{U}_0$. Again, this is because $\bsingle$ vanishes at $\df$. But this integral is precisely
 \[ \Big(\pi_l^\partial\big|_{B_0}\Big)_* \left( \bcalN \,\cdot\, (\pi_r^\partial)^*(f) \right) (y) \;, \]
where $B_0 = (\pi_l^\partial)^{-1}(y) \cap \mathcal{U}_0$ because $\bcalN$ is the leading order term of $\bsingle$ at $\lf(Y)$.

\myskip Now since $B_0 \cup B_2 = (\pi_l^\partial)^{-1}(y)$, we obtain the claim as long as we integrate over fibres that stay away from $\bff$ hence on $\mathring{Y}$. Near $\bff$, we need to use different sets of coordinates, but the calculations work in just the same way proving the claim on $Y$.
\end{proof}

\begin{lemma}\label{thm:jump.lemma.2} The leading order term of $\Op{\double}f$ at $Y$ is given by $\tfrac{1}{2}f$. More precisely, given a product neighbourhood $[0,1) \times Y$ of $Y \subset X$, there is a function $R \in \dot{C}^0\big([0,1);\calA(Y)\big)$ so that, on this product neighbourhood,
 \[ \Op{\double}f (\chi,y) = \tfrac{1}{2}f(y) + R(\chi,y) \;. \]
\end{lemma}

\begin{proof} We proceed in the exact same way as in the proof of \Cref{thm:jump.lemma.1}: As for the integral over $A_2$, since $\bdouble$ has index set $1$ at $\lf(Y)$, we obtain
 \[ \Big(\pi_l\big|_{A_2}\Big)_* \left( \bdouble \,\cdot\, \pi_{r}^*(f) \right)(\chi,y) \too 0 \quad\text{as $\chi \to 0$.} \]
Using the same local coordinates $\big(\chi,\tfrac{s}{\chi},\eta,y\big)$ as before, on $A_1$ we obtain
 \begin{equation}\label{eq:double.local.int}
  \preb{\double} \,\cdot\, \pi_r^*(f) = \tfrac{1}{\vol\Sph^{n-1}} \, \frac{(\frac{s}{\chi})^{n-1}}{(1+(\frac{s}{\chi})^2)^{\tfrac{n}{2}}} f\big(y-\chi (\tfrac{s}{\chi})\eta\big) \bfracdens{s}{\chi} d\eta \pi_l^*(\mu)
 \end{equation}
and interchanging the order of integration and taking the limit, we obtain part of the contribution from $\df$ to the leading order term:
 \begin{equation}
  \Big(\pi_l\big|_{A_1}\Big)_* \left( \bdouble \,\cdot\, \pi_{r}^*(f) \right)(\chi,y)
  \xrightarrow{\,\chi \to 0\,} \left( \frac{\vol\Sph^{n-2}}{\vol\Sph^{n-1}}
         \int_0^c \frac{(\tfrac{s}{\chi})^{n-1}}{(1+(\tfrac{s}{\chi})^2)^{\tfrac{n}{2}}} \bfracdens{s}{\chi} \right)
         f(y) \big|dy\big| \;,
 \end{equation}
where $c$ is chosen according to the trivialisation $\pi_l^{-1}(0,y) \cap \mathcal{U}_1 \cong Y \times [0,c) \times \Sph^{n-2}$. Still following the lines of the proof of \Cref{thm:jump.lemma.1}, we see that the contribution from $A_0$ to the leading order term is given by integrating over $A_0^\pp = \pi_l^{-1}(0,y) \cap \df \cap \mathcal{U}_0$. Using coordinates $\big(\tfrac{\chi}{s}, s, \eta, y\big)$ and setting $s = 0$, we arrive at
 \[ \Big(\pi_l\big|_{A_0}\Big)_* \left( \bdouble \,\cdot\, \pi_{r}^*(f) \right)(\chi,y)
  \xrightarrow{\,\chi \to 0\,} \left( \frac{\vol\Sph^{n-2}}{\vol\Sph^{n-1}} \int_0^{c^{-1}} \frac{ (\frac{\chi}{s})^{n-1}}{\big(1+(\frac{\chi}{s})^2\big)^{\tfrac{n}{2}}} \bfracdens{\chi}{s} \right) f(y) |dy| \;, \]
with the same constant $c$ as before. Now, combining the contributions from $A_0$ and $A_1$, we get
 \begin{align*}
  \lim_{\chi \to 0} \Op{\double}f (\chi,y)
   &= \left( \frac{\vol\Sph^{n-2}}{\vol\Sph^{n-1}} \int_0^\infty \frac{t^{n-1}}{(1+t^2)^{\tfrac{n}{2}}} \bdens{t} \right) f(y) |dy| \\
   &= \left( \frac{\vol\Sph^{n-2}}{\vol\Sph^{n-1}} \frac{\sqrt{\pi} \, \Gamma(\frac{n-1}{2})}{2 \, \Gamma(\frac{n}{2})} \right) f(y) |dy|
   = \tfrac{1}{2} f(y) |dy| \;.
 \end{align*}
where in this single instance, $\Gamma$ denotes the Gamma-function. As $A_1 \cup A_2 = \pi_{l}^{-1}(\chi,y)$, we have just shown that the leading order term of $\Op{\double}f$ at $\mathring{Y}$ is indeed, after identifying it with a function via \Cref{eq:fct.dens.ident.detail.1}, $\tfrac{1}{2}f$. Performing these calculations in local coordinates valid near $\bff \cap \df$ (for instance, using $\big( xr, \tfrac{s}{xr}, \eta, y\big)$) will give the same limit as we approach $Y \subset X$. This completes the proof.
\end{proof}

For the following lemma, relating the subleading order term of $\Op{\single}f$ to the leading order term of $\Op{\double}f$, we need to be aware that only the constant term in a polyhomogeneous expansion is defined independently of the choice of boundary defining function. This is why we formulate the next result in specific coordinates, only.

\begin{lemma}\label{thm:jump.lemma.3} The leading order term of $\partial_\nu \Op{\single}f$ at $Y$ is $\tfrac{1}{2}f$. In particular, the subleading order term of\, $\Op{\single}f$ at $Y$ is $\tfrac{1}{2}f(r,\omega)$ with respect to $x$ and $\tfrac{1}{2r} f(r,\omega)$ with respect to $xr$.
\end{lemma}

\begin{proof} Since $\single$ is even and $\scal{\nu(z^\p)}{z-z^\p}$ is odd with respect to interchanging $z$ and $z^\p$, we have equality of kernels
 \[ \big(\partial_\nu \single\big)(z,z^\p) = \big(\partial_{\nu^\p} \single\big)(z^\p,z)
    = \double (z^\p,z) = -\double(z,z^\p) \]
and hence $\partial_\nu \big[\Op{\single}f\big] = -\Op{\double}f$. Consequently, both sides have the same leading order term. Noting that $\partial_\nu = -\tfrac{1}{x} (x\partial_x) = -\tfrac{r}{xr}(xr\partial_{xr})$, a comparison of the expansions of both layer potentials in local coordinates $(x,r,\omega)$ and $(xr,r,\omega)$ completes the proof.
\end{proof}

\begin{lemma}\label{thm:jump.lemma.4} \Crefrange{thm:jump.lemma.1}{thm:jump.lemma.3} hold for the modified layer potentials as well.
\end{lemma}

\begin{proof} The modification terms (as $b$-densities) for the single layer potential have index sets $0$ and $n-1$ at $\lf(Y)$ and $\df$, and we have uniform convergence
 \[ \big(\bsingle_k\,\pi_r^*(f)\big)(\chi,y,y^\p) \too \big(\bcalN_k\,\pi_r^*(f)\big)(y,y^\p) \]
for $(y,y^\p)$ in a compact subset of $\lf(Y)$. Therefore, \Cref{thm:jump.lemma.1} continues to hold with $\single$ and $\calN$ replaced by $\single_k$ and $\calN_k$. The modification terms for the double layer potential have index sets $1$ and $n$ at $\lf(Y)$ and $\df$ and consequently do not contribute to the leading order term at all, whence \Cref{thm:jump.lemma.2} holds with $\double$ replaced by $\double_k$. As for \Cref{thm:jump.lemma.3}, we have that
 \[ \partial_x \left( |z|^m C^{\frac{n-2}{2}}_m(\Theta)\middle) \right|_{x=0}
     = x \left( m |y|^{m-2} C^{\frac{n-2}{2}}_m(\Theta^\p) + (n-2) \Theta^\p |y|^{m} C^{\frac{n}{2}}_{m-1}(\Theta^\p) \middle) \right|_{x=0} = 0 \;, \]
using $\partial_t C^\lambda_m(t) = 2\lambda C^{\lambda+1}_{m-1}(t)$ (see \cite[(4.7.14)]{Sze39}) and where we wrote $\Theta^\p = \Theta\big|_{x=0}$. Thus, the modification terms do not contribute to the leading order term of $\partial_\nu\Op{\single_k}f$ which is to say that \Cref{thm:jump.lemma.3} holds for the modified layer potentials as well.
\end{proof}

\section{Background Material}\label{APP.sec:defs}

The following material can be found in different forms and flavours in various sources, for instance in \cite{Mel92,Mel93,Mel96,Loy98,MM98,Gri01,GH09,GH14}. We will in particular make use of \cite{Mel96} and \cite{GH09,GH14}.

\subsection{Manifolds with Corners}\label{APP.ssec:mwcs}

In the same way in which a manifold is modeled over open subsets of $\R^n$, manifolds with corners are modeled over relatively open subsets of
 \begin{equation}\label{APP.eq:Rnk}
  \R^n_k = \big\{\, x \in \R^n \,\big|\, x_1,\dotsc x_k \ge 0 \,\big\} \;,
 \end{equation}
i.e., over restrictions to $\R^n_k$ of open subsets of $\R^n$. A function $f : \Omega \too \Omega^\p$ between relatively open subsets of $\R^n_k$ respectively $\R^m_l$ is called \emph{smooth (up to the boundary)}, if there is a smooth extension $\tilde{f} : \R^n \too \R^m$ of $f$, or, equivalently, if it is smooth in the interior $\mathring{\Omega}$ of $\Omega$ with all derivatives bounded on compact subsets of $\Omega$.

\begin{define}\label{APP.def:mwc} A \emph{manifold with corners} is a second countable Hausdorff space $X$ so that there is an open cover $\big\{\mathcal{U}_i\big\}_{i\in\N}$ of $X$ and a family of homeomorphisms $\varphi_i : \mathcal{U}_i \too U_i$, $U_i \subset \R^n_{k_i}$ open for some $0\le k_i \le n$, such that for all $i$, $j$, the coordinate changes $\varphi_i \circ \varphi_j^{-1}$ are smooth up to the boundary and where all boundary hypersurfaces are embedded.
\end{define}
Any manifold with corners has, locally, \emph{product coordinates} $(x,y) \in \R^k_+ \times \R^{n-k} = \R^n_k$, where the vanishing of the $x_i$ determines the boundary hypersurfaces. To precisely define the term \emph{boundary hypersurface}, consider coordinates \emph{based at $p\in X$}, i.e., a chart $(\mathcal{U},\varphi)$ such that $\varphi(p) = 0$. This in fact fixes a minimal $k$ for which a neighbourhood of $p$ can be mapped diffeomorphically onto an open subset of $\R^n_k$. This minimal $k$ is called the \emph{codimension of $p$ in $X$}. Then,
 \[ \partial_k X = \big\{\, p \in X \,\big|\, \text{$p$ has codimension $k$ in $X$} \,\big\} \;, \]
is independent of the choice of the chart. We call the closure of a connected component of $\partial_k X$ a \emph{boundary face of codimension $k$} of $X$, or a \emph{boundary hypersurface} in the case $k=1$. The set of codimension $k$-faces of $X$ will be denoted by $\calM_k(X)$, $\calM(X)$ is the set of all boundary faces.

Let $H$ be a boundary hypersurface of a manifold with corners. A \emph{boundary defining function} for $H$ is a smooth, non-negative function $\rho$ defined on a neighbourhood of $H$ such that $H = \rho^{-1}(0)$ and near each point $p$ in $H$, there are local coordinates with $\rho$ as first element (which is equivalent to saying that $d\rho \neq 0$ at $H$). In fact, on a manifold with corners, we may always assume that boundary defining functions are globally defined.

%

There are multiple possible notions of submanifolds of manifolds with corners, see \cite[1.7 ff.]{Mel96}, but we will mostly use only one of them $p$-submanifolds or \emph{product submanifolds}. Product manifolds can best be defined using local coordinates. Away from the boundary, these are just submanifolds in the usual sense, whereas near the boundary, if a manifold with corners can locally be identified with a neighbourhood of the origin in $\R^n_k$, then a $p$-submanifold corresponds to the intersection of one or more of the coordinate planes. The following precise definition is taken from \cite{GH09}.

\begin{define}\label{APP.def:p.sub} A \emph{p-submanifold} of a manifold with corners $X$ is a subset $Y \subset X$ that satisfies the following local product condition: For each $p \in Y$ there is a neighbourhood $O$ of $p$ in $X$ and coordinates $x_1,\dotsc,x_k,y_{k+1},\dotsc,y_{n-k}$ such that $Y$ is given by the vanishing of some subset of these coordinates. Here, $k$ is the maximal codimension of $O$ and $x_1,\dotsc,x_k$ are boundary defining functions for the boundary hypersurfaces of $O$. If $Y$ is given everywhere locally by the vanishing of a subset of the $y_j$, we call it an \emph{interior p-submanifold}, otherwise we call it a \emph{boundary p-submanifold}.
\end{define}

Using the definitions, it is easy to see that boundary faces of a manifold with corners are always p-submanifolds and that any p-submanifold of a manifold with corners is a manifold with corners in its own right. In particular, if $Y$ is a $p$-submanifold of $X$ and $p \in Y$, we may choose \emph{adapted coordinates} near $p$ by taking the coordinates from \Cref{APP.def:p.sub} and refining it by using boundary defining functions.

\subsection{Polyhomogeneity}\label{APP.ssec:phgy}

In the following, we define spaces of functions with prescribed behaviour near boundary faces of manifolds with corners. The definitions from this subsection are taken from \cite{GH09}, but originate from \cite{Mel92}.

\begin{define}[Index Sets]\label{APP.def:index.set} An \emph{index set} at a boundary hypersurface $H$ of a manifold with corners $X$ is a discrete subset $G \subset \C \times \N_0$ satisfying:
 \begin{enumerate}
  \item For every $c \in \R$, the set $G \cap \big(\{\Re z < c \} \times \N_0\big)$ is finite
  \item If $(z,p) \in G$ and $0 \le q \le p$, then $(z,q) \in G$
  \item If $(z,p) \in G$, then also $(z+1,p) \in G$
 \end{enumerate}
An \emph{index family} $\calG$ for $X$ is a choice of index sets $\calG(H)$ for each boundary hypersurface $H$ of $X$.
\end{define}
There are some index sets deserving their own notation. The sets
 \begin{equation}\label{APP.eq:integer.index.set}
  k = \myset{(z,0)}{z \ge k}
 \end{equation}
are certainly index sets, as is the empty set. We call these \emph{integer index sets}. If $\mathcal{G}$ is an index family on a manifold with corners and $H$ is a boundary hypersurface of $X$, we denote by $\mathcal{G}_H$ the index family for $H$ given by the index sets $\calG(H^\p)$, where $H^\p$ is a boundary hypersurface of $X$ and $H^\p \cap H \neq \emptyset$. Given index families $\calG = \{ G_1,\dotsc,G_N \}$ and $\calG^\p = \{ G_1^\p,\dotsc,G_N^\p \}$, we may construct new index families by putting
 \begin{align*}
   G_j + G_j^\p &:=
    \big\{\, (z+z^\p, \max (p,p^\p) \,\big|\, (z,p) \in G_j, \, (z^\p,p^\p) \in G_j^\p \,\big\} \text{ and} \\
   \calG + \calG^\p &:= \big( G_1 + G_1^\p,\dotsc, G_N+G_N^\p \big) \;, \\
   \calG \cup \calG^\p &:= \big( G_1 \cup G_1^\p, \dotsc, G_N \cup G_N^\p \big) \;, \\
   \calG \cap \calG^\p &:= \big( G_1 \cap G_1^\p, \dotsc, G_N \cap G_N^\p \big) \;.
 \end{align*}
We will also need the \emph{extended union} of two index sets $G_1$, $G_2$. It is constructed from the union of the index sets by adding certain exponents, which will correspond to logarithmic terms in polyhomogeneous expansions as is defined below:
 \begin{equation}\label{APP.def:extended.union}
  G_1 \bcup G_2 = G_1 \cup G_2 \cup \myset{(z,p_1+p_2+1)}{(z,p_1) \in G_1,\, (z,p_2) \in G_2}
 \end{equation}
To conclude, we define the \emph{real part} of an index set $G$ by
 \begin{equation}\label{APP.def:real.part}
  \Re G = \min\myset{\Re z}{(z,p) \in G} \;.
 \end{equation}
and introduce a partial ordering on the set of index sets by saying that $G_1 < G_2$ if and only if $\Re G_1 < \Re G_2$.

Now we may introduce spaces of functions whose boundary behavior is given by index families. We proceed by defining \emph{polyhomogeneity} on manifolds with boundary first, and then generalising to manifolds with corners. The spaces $\dot{C}^N(X)$ denote functions on $X$ that are smooth up to the boundary and whose first $N$ derivatives vanish at $\partial X$.

\begin{define}[Polyhomogeneous Functions]\label{APP.def:phg} Let $X$ be a manifold with boundary, $x$ a boundary defining function on $X$ and $G$ be an index set for $\partial X$. A function $u$ on $X$ is called \emph{polyhomogeneous with index set $G$ at $H$}, $u \in \calA^G(X)$, if $u \in C^\infty(\mathring{X})$ and if, for $(z,p) \in G$, there are functions $u_{z,p} \in C^\infty(X)$ such that
 \[
  u - \sum_{\substack{(z,p)\in G \\ \Re z \le N }} x^z (\log x)^p u_{z,p} \, \in \, \dot{C}^N(X).
 \]
We denote this by writing
 \begin{equation}\label{APP.eq:phg.expansion}
  u \sim \sum_{(z,p) \in G} x^z (\log x)^p u_{z,p} \;.
 \end{equation}
If $X$ is a manifold with corners and $\mathcal{G}$ is an index family for $X$, a function $u$ on $X$ is called \emph{polyhomogeneous with index family} $\mathcal{G}$, $u \in \calA^{\mathcal{G}}(X)$, if $u \in C^\infty(\mathring{X})$ and near each boundary hypersurface $H$ of $X$, $u$ has an asymptotic expansion as in \cref{APP.eq:phg.expansion}, where $x$ is a boundary defining function for $H$ and
 \begin{enumerate}
  \item the coefficients $u_{z,p}$ lie in $C^\infty\big([0,1);\calA^{\mathcal{G}_H}(H)\big)$ and
  \item the remainder is in $\dot{C}^N\big([0,1);\calA^{\mathcal{G}_H}(H)\big)$,
 \end{enumerate}
both with respect to some trivialisation $[0,1)\times H$ of a neighbourhood of $H$.
\end{define}

For instance, on a manifold with boundary, the index set $G=0$ corresponds to functions that are smooth up to the boundary, $\calA^0(X) = C^\infty(X)$, the empty index set corresponds to smooth functions that vanish to any order at the boundary, $\calA^\emptyset(X) = \dot{C}^\infty(X)$. For this reason, we will also denote the empty index set by $\infty$. We will mostly deal with index sets of the type $G = k$, the space $\calA^k(X)$ comprises of those functions that have a Laurent series in terms of a boundary defining function, where the lowest order term is exactly the $k$-th power of this boundary defining function. In particular, for $k \ge 0$, we have $\calA^{k+1}(X) \subset \dot{C}^k(X)$.

We also want to note the following properties of spaces of polyhomogeneous functions, they follow directly from the definition:
 \begin{align*}
   C^\infty(X) \cdot \calA^\calG(X) &\subset \calA^\calG(X) &
   \calA^\calG(X) + \calA^{\calG^\p}(X) &\subset \calA^{\calG \cup \calG^\p}(X) \\
   \calA^\calG(X) \cdot \calA^{\calG^\p}(X) &\subset \calA^{\calG+\calG^\p}(X) &
   \calA^\calG(X) \cap \calA^{\calG^\p}(X) &\subset \calA^{\calG \cap \calG^\p}(X)
 \end{align*}
With \Cref{APP.def:phg} in mind, the definition of index sets should also become clearer: Condition \emph*{i)} in \Cref{APP.def:index.set} ensures that the asymptotic sum \cref{APP.eq:phg.expansion} makes sense, condition \emph*{iii)} ensures invariance with respect to coordinate changes. Condition \emph*{ii)} on the other hand is necessary for $\calA^\calG(X)$ to be closed under the action of $b$-differential operators.

Polyhomogeneity can be readily generalised to $b$-densities as introduced in \cite{Mel93} and \cite{Mel96}. We do not go into the details, but say the following: Given $s \in \R$, the space of $s$-densities on a vector space $V$ of dimension $n$ is
 \begin{equation}\label{APP.def:dens}
  \Omega^s(V) = \Myset{ \mu : \Lambda^n V\setminus \{0\} \too \R }{ \mu(\lambda \alpha) = |\lambda|^s\mu(\alpha) \text{ for all } \alpha \in \Lambda^n V\setminus \{0\}, \lambda \neq 0 } \;.
 \end{equation}
The space $\Omega^s(V)$ is one-dimensional and spanned by $|dv_1 \wedge \dotsc \wedge dv_n|^s$, where $(v_1,\dotsc,v_n)$ is a basis for $V$. If $s=1$, we simply write $\Omega(V)$. Using $V = TX$, this yields densities defined on a manifold. If $X$ is a manifold with corners, the subset of vector fields that are tangent to the boundary, $\mathcal{V}_b$, is in fact the set of sections of a vector bundle, the \emph{$b$-tangent bundle} $\preb{T}X$.

\begin{define}[$b$-Densities]\label{APP.def:b.dens} The space of \emph{$b$-densities} and \emph{$b$-half densities} on a manifold with corners $X$ is the space of smooth sections of the bundle $\Omega(\preb{T}X)$ respectively $\Omega^{\frac{1}{2}}(\preb{T}X)$, denoted by $C^\infty(X;\bOmega)$ or $C^\infty(X;\bOmegah)$, equipped with the topology of uniform convergence of all derivatives on compact subsets. Spaces of polyhomogeneous $b$-densities will be correspondingly denoted by $\calA^\calG(X;\bOmega)$ respectively $\calA^\calG(X;\bOmegah)$.
\end{define}
Away from the boundary, a $b$-density is just a smooth density. If $p \in \partial X$ and $(x,y)$ are local coordinates near $p$ as in \Cref{APP.def:p.sub}, near $p$ any $\mu \in C^\infty(X;\bOmega^s)$ is of the form
 \begin{equation}\label{APP.eq:b.dens}
  \mu = x_1^s \dotsm x_k^s \big|\tfrac{dx}{x}\big|^s \, \big|dy\big|^s \;,
 \end{equation}
where we write $\big|\tfrac{dx}{x}\big|$ for $\big|\tfrac{dx_1}{x_1} \dotsm \tfrac{dx_k}{x_k}\big|$ and similarly for $\big|dy\big|$. Observe that 1-densities can be invariantly integrated and that $C^\infty(X;\bOmegah) \otimes C^\infty(X;\bOmegah)$ is canonically isomorphic to $C^\infty(X;\bOmega)$, which is to say that the product of two half-densities can be invariantly integrated. This explains the importance of $b$-half densities.

\subsection{\textb-Maps and \textb-Fibrations}\label{APP.ssec:maps.fibrations}

Each category of spaces with a specific structure has its maps respecting this structure. In the case of manifolds with boundary, these are the $b$-maps:

\begin{define}\label{APP.def:b.map} Let $X$ and $Y$ be manifolds with corners and $\{\rho_G\}$, $\{\eta_H\}$, be complete sets of boundary defining functions for the boundary hypersurfaces of $X$ and $Y$, respectively. $F \in C^\infty(X,Y)$ is called a \emph{$b$-map} if, for each $H \in \calM_1(Y)$, either $\eta_H \circ F = 0$ or
 \begin{equation}\label{APP.eq:b.map}
  (\eta_H \circ F) = a_H \prod_{G \in \calM_1(X)} \rho_G^{e_F(G,H)} \;,
 \end{equation}
for smooth, positive functions $a_H$ and non-negative integers $e_F(G,H)$. We say that $F$ is an \emph{interior $b$-map}, if $\eta_H \circ F \neq 0$ for all $H \in \calM_1(Y)$ and else call $F$ a \emph{boundary $b$-map}. The integers $e_F(G,H)$ are called \emph{boundary exponents} and their collection is called the \emph{exponent matrix} of $F$. The set of boundary hypersurfaces $G$ of $X$ such that $e_F(G,H)=0$ for all $H \in \calM_1(Y)$ is called the \emph{null set} of $e_F$, $\nul(e_F)$.

\end{define}

Note that $e_F(G,H) > 0$ if and only if $F(G) \subset H$ and that any $b$-map necessarily maps boundary faces to boundary faces and hence induces a map $\overline{F} : \calM(X) \too \calM(Y)$. This leads to a geometric definition of $b$-fibrations, cf.~\cite{Gri01}. As is pointed out there, it is enough to require the first item for boundary hypersurfaces.

\begin{define}\label{APP.def:b.fibration} Let $X$ and $Y$ be manifolds with corners and $F \in C^\infty(X,Y)$ be a $b$-map. $F$ is called a $b$-\emph{fibration}, if, for any $W \in \calM(X)$,
\begin{enumerate}
 \item $\codim \overline{F}(W) \le \codim W$
 \item $F$ is a fibration $\mathring{W} \too \big(\overline{F}(W)\big)\mathring{}$.
\end{enumerate}
\end{define}

At the heart of the calculus of conormal distributions on manifolds with corners are the twin theorems on the behavior of pull-backs of functions and push-forwards of densities. The following formulations are adapted from \cite{Gri01}, see \cite{Mel92}, \cite{Mel93} or \cite{Mel96} for proofs.

\begin{thm}[Pull-Back Theorem]\label{APP.thm:pbt} Let $F : X \too Y$ be a $b$-map and $f \in \calA^\calG(Y)$. Then, $F^*f$ is polyhomogeneous on $X$ with index family $F^\sharp\calG$ defined as follows: If $G \in \calM_1(X)$, then
 \[ F^\sharp\calG(G)
    = \left\{\, \Big(q + \sum_H e_F(G,H) z_H, p_H\Big) \,\right\} \;, \]
where $q \in \N_0$, $H \in \calM_1(Y)$ and either $(z_H,p_H) = (0,0)$ or $F(G)\subset H$ and $(z_H,p_H) \in \calG(H)$.
\end{thm}

\begin{thm}[Push-Forward Theorem]\label{APP.thm:pft} Let $F : X \too Y$ be a $b$-fibration of compact manifolds with corners. If $\calE$ is an index family for $X$ such that
 \begin{equation}\label{APP.eq:pft.int.cond}
  G \in \nul(e_F) \quad\Longrightarrow\quad \Re\calE(G) > 0 \;,
 \end{equation}
and if $u \in \calA^\calG(X;\bOmega)$, then $F_*u$ is a polyhomogeneous $b$-density on $Y$ with index family given by
 \begin{equation}\label{APP.eq:pft.index}
  F_\sharp\calE(H) \deq \Bcup_G \Myset{\left( \frac{z}{e_F(G,H)},p\right)}{(z,p)\in\calE(G)} \;,
 \end{equation}
where the extended union is over boundary hypersurfaces $G \in \calM_1(X)$ with $F(G) \subset H$.
\end{thm}

Note that the definition of $F^\sharp\calG$ does not involve the extended union while the definition of $F_\sharp\calE$ does: Consequently, a pull-back cannot create additional logarithmic terms whereas a push-forward can. Moreover, the index set $F_\sharp\calE(H)$ is known to be too big: In \cite{Gri01}, a refined index set, say $F_\sharp^\p\calE(H)$, is defined which in some cases is strictly smaller yet \Cref{APP.thm:pft} is still satisfied. In the situation described here, the sets coincide for any $H$. Nevertheless, in both cases, the additional logarithmic terms might not appear, depending on the specific density $u$.

\subsection{Blow-Ups of Product Submanifolds}\label{APP.ssec:blow.ups}

Real (radial) blow-ups are an indispensable tool in singular geometric analysis. They can be used to resolve singularities in a very intuitive, geometric way. For more details and proofs, we refer to \cite{Gri01}, \cite{Mel93} and \cite{Mel96}. The basic example is simply introduction of polar coordinates in $\R^n$:
 \begin{equation}\label{APP.eq:simple.blowup}
  \beta : \R_+ \times \Sph^{n-1} \ni (r,\omega) \mapstoo (r\omega) \in \R^n
 \end{equation}
This is in fact a smooth map which restricts to a diffeomorphism $(0,\infty) \times \Sph^{n-1} \too \R^n\setminus\{0\}$. The blow-up, the domain of $\beta$, is denoted by $[\R^n;\{0\}]$ and $\beta$ is called the \emph{blow-down map}. The boundary $\{0\}\times\Sph^{n-1}$ is called the \emph{front face} and often denoted by $\ff(\beta)$. Please note that $\beta(\ff(\beta)) = \{0\}$. More generally, one goes over to define the blow-up of the origin of any vector space, then of the zero section of a vector bundle, and finally the blow-up of a $p$-submanifold $Y$ of a manifold with corners $X$: The front face of the blow-up \cref{APP.eq:simple.blowup} is diffeomorphic to $\Sph^{n-1}$, which is (in this case) just the set of unit tangent vectors in $\R^n$ that are inward pointing and normal to the origin, i.e., the inward pointing spherical normal bundle $\Sph_+ N \{0\}$. Although we will not go into the details here, this bundle can be defined for any $p$-submanifold $Y$ of a manifold with corners $X$ as well, and we define (cf.~\cite[5.3]{Mel96}),
 \begin{equation}\label{APP.eq:general.blowup}
  \big[X;Y\big] = \Sph_+N Y \sqcup X \setminus Y \;,
 \end{equation}
where $\sqcup$ denotes a disjoint union. In Section 5.3 of \cite{Mel96} it is also shown that if $Y$ is a $p$-submanifold of $X$, $[X;Y]$ naturally inherits the structure of a manifold with corners from $X$ and $\beta$ is a $b$-map and a diffeomorphism away from $\ff(\beta)$. Moreover, the smooth vector fields on $X$, tangent to $Y$, lift to span the smooth vector fields on $[X;Y]$, tangent to $\ff(\beta)$.

We will denote iterated blow-ups by $[X;Y_1;\dotsc;Y_k]$ for instance. Please observe that if the order of blow-ups is changed, the results will not be diffeomorphic spaces, even if the altered blow-up is still well-defined. There are certain cases in which the order of blow-ups does not make a difference, compare \cite[5.8]{Mel96}, the most prominent ones being the cases of nested $p$-submanifolds and of $p$-submanifolds that meet transversally.

If $(x,y)=(x_1,\dotsc,x_k,y_1,\dotsc,y_l)$ are local product coordinates of $X$ near a point $p \in Y$ so that, locally, $Y= \{y_1=\dotsc=y_l=0\}$, we may use these coordinates to obtain \emph{projective coordinates} for a neighbourhood of $q \in [X;Y]$, $\beta(q)=p$. For at least one of the coordinates $y_j$, and without loss of generality we assume it is $y_1$, we have $\beta^*(dy_1) \neq 0$ at $q$. Then, the set of functions
 \begin{equation}\label{APP.def:proj.coords}
  \big(x_1, \dotsc, x_k, y_1, \tfrac{y_2}{y_1},\dotsc,\tfrac{y_l}{y_1} \big)
 \end{equation}
lift to give local coordinates near $q$. In particular, $y_1$ lifts to a boundary defining function of $\ff(\beta)$ near $q$.

To conclude this section, let us remark that at least in the case of blow-down maps, the pull-back of smooth densities is well-defined, compare \cite[5.5]{Mel96} or \cite[4.7]{Mel93}. This is so because $\beta$ restricts to a diffeomorphism $[X,Y]\mathring{\phantom{l}} \too \mathring{X}$ and we can lift a distributional density on $X$ to act on a compactly supported function $f$ on $[X,Y]$ by duality, i.e., on $(\beta^{-1})^*f$.


\subsection{\textb-Pseudodifferential Operators}\label{APP.sec:pseudos}

For definitions and results on pseudodifferential operators on $\R^n$ and manifolds without boundary, please see \cite{Hoer85,Shu87,Tay81}. $b$-pseudodifferential operators are best defined in terms of their kernels, which are requested to be polyhomogeneous conormal $b$-half densities on the $b$-double space. The details are given in \cite{Mel93,Mel96}, for instance, and \cite{Gri01} gives a good to read introduction to the $b$-calculus.

Let $X$ be a manifold with boundary and assume for simplicity (and because that is the situation we will need) that $\partial X$ is connected. (The general case can be found in \cite{Loy98}.) The $b$-double space of $X$ is the blow-up
 \begin{equation}\label{APP.def:b.double}
  X_b^2 = \big[X^2;(\partial X)^2\big] \xrightarrow{\quad \beta_b \quad} X^2 \;,
 \end{equation}
that is, we take its double space and blow up the single corner it has. $X_b^2$ has three boundary hypersurfaces: The \emph{left face} $\lf = \beta_b^\ast(\partial X \times X)$, the \emph{right face} $\rf = \beta_b^\ast(X \times \partial X)$ and the \emph{front face} $\ff = \beta_b^\ast(\partial X \times \partial X)$. Another important submanifold is the lift $D_b$ of the diagonal $D \subset X^2$, it is an interior $p$-submanifold of $X_b^2$ and has a well-defined normal bundle in $X_b^2$, up to the boundary. The latter allows us to extend the concept of conormality to $X_b^2$. First of all, we recall the definition of classical (or one step polyhomogeneous) conormal distributions on manifolds without boundary:

\begin{define}[Conormal Distributions]\label{APP.def:conormal} Let $X$ be a manifold and $Y \subset X$ be an embedded submanifold. A distribution $u \in C^{-\infty}(X)$ is called a \emph{(classical) conormal distribution of degree} $m \in \R$ \emph{at} $Y$ if $u \in C^\infty(X \setminus Y)$ and, given local coordinates $y_1,\dotsc,y_k,z_1,\dotsc,z_l$ around $Y$ such that $Y$ is given by $z_1=\dotsc=z_l = 0$, we have
 \begin{equation}\label{APP.eq:conormal.fourier}
  u(y,z) = \int_{\R^l} e^{iz\zeta} \sigma(y,\zeta) d\zeta \;,
 \end{equation}
where $\sigma$ has an asymptotic expansion
 \begin{equation}\label{APP.eq:conormal.sum}
  \sigma(y,\zeta) \sim \sum_{j=0}^\infty \sigma_{m^\p-j}(y,\zeta) \;,
 \end{equation}
with the $\sigma_{m^\p-j}$ being homogeneous of degree $m^\p-j$ in $\zeta$, where $m^\p = m + \frac{1}{4}\dim X - \frac{1}{2}\codim Y$. The space of distributions that are conormal of degree $m$ with respect to $Y$ is denoted by $I^m(X,Y)$.
\end{define}

The above asymptotic expansion is meant to be in terms of symbols, which is to say that for any $N \in \N$ we assume $\sigma - \sum_{j=0}^N \sigma_{m^\p-j} \in S^{m^\p-N-1}(Y \times \R^l)$, compare \cite[3.3]{Shu87}. Thus, $\sigma$ is in fact a \emph{classical symbol}, $\sigma \in S_{cl}^{m^\p}(Y \times \R^l)$. There is a different, yet equivalent definition of conormality. In \cite[Ch. 6.1]{Mel96} the following space is introduced:
 \begin{equation}\label{APP.eq:alt.conormal}
  I^{[s]}(X,Y) = \big\{\, u \in H^s(X) \,\big|\,
                           \text{for all}\, k \in \N_0 : \mathcal{V}(X,Y)^k u \subset H^s(X) \,\big\} \;,
 \end{equation}
where $\mathcal{V}(X,Y)$ is the space of smooth vector fields on $X$ that are tangent to $Y$, i.e., those fields $V \in C^\infty(X;TX)$ such that $V(x) \in T_xY$ for all $x \in Y$. Melrose then shows that these spaces actually coincide, apart from a sign in the degree, $I^{[-s]}(X,Y) = I^s(X,Y)$. (Compare also \cite[Ch. 18.2]{Hoer85}.)

Now we combine \Cref{APP.def:phg,APP.def:conormal} to define conormality on manifolds with corners. The resulting space of distributions will be used to define the calculus of $b$-pseudodifferential operators.

\begin{define}[Polyhomogeneous Conormal Distributions]\label{APP.def:phg.conormal} Let $X$ be a manifold with corners, $\mathcal{G}$ be an index family for $X$ and let $Y$ be an interior p-submanifold of $X$. A distribution $u$ on $X$ is both polyhomogeneous with index family $\mathcal{G}$ at $\partial X$ and conormal of degree $m$ with respect to $Y$, if $u \in I^m(\mathring{X},\mathring{Y})$ and, near each boundary hypersurface $H$ of $X$ with boundary defining function $x$, $u$ has an asymptotic expansion as in \emph{\cref{APP.eq:phg.expansion}} where
 \begin{enumerate}
  \item the coefficients $u_{z,p}$ lie in $C^\infty\big([0,1) ; \calI^{m+\frac{1}{4},\mathcal{G}_H}(H,H \cap Y) \big)$ and
  \item the remainder is in $\dot{C}^N\big( [0,1) ; \calI^{m+\frac{1}{4},\mathcal{G}_H}(H,H \cap Y) \big)$,
 \end{enumerate}
both with respect to a trivialisation $[0,1)\times H$ of a neighbourhood of $H$. We denote the space of polyhomogeneous conormal distributions by $\calI^{m,\mathcal{G}}(X,Y)$.
\end{define}

Please observe that \Cref{APP.def:phg.conormal} can readily be extended to sections of vector bundles, most importantly to half-densities, by requiring the coefficients in any local trivialisation to satisfy the respective conditions. Moreover, these spaces are $C^\infty_c(X)$-modules (compare \cite{Fri,Mel96}).

In order to define $b$-pseudodifferential operators on $X$, let $\pi_{j,b} : X_b^2 \too X$ be the composition of the blow-down map \cref{APP.def:b.double} and projection onto the $j$-th factor. Then, a $b$-half density $A$ on $X_b^2$ acts on a $b$-half density $u$ on $X$ by
 \begin{equation}\label{APP.eq:b.action}
  \Op{A}u = (\pi_{1,b})_*\left(A \, (\pi_{2,b})^*u\right) \;,
 \end{equation}
whenever the push-forward converges and the right-hand side of \Cref{APP.eq:b.action} is again a $b$-half density.

\begin{define}[\textb-Pseudodifferential Operators]\label{APP.def:b.pseudo} Let $X$ be a compact manifold with boundary. The \emph{small calculus} of $b$-pseudo\-differential operators of degree $m$ on $X$ is the space $\calI^{m,(\infty,\infty,0)}(X_b^2,D_b;\bOmegah)$ and is denoted by $\Psi_b^m(X)$. Elements of $\Psi_b^m(X)$ act on $b$-half-densities on $X$ by \emph{\cref{APP.eq:b.action}}, where the index family $(\infty,\infty,0)$ refers to the ordering $(\lf,\rf,\ff)$. The \emph{full calculus} $\Psi_b^{m,\calG}(X)$ is the space $\calI^{m,\calG}(X_b^2,D_b;\bOmegah)$.
\end{define}

As is shown in \cite[Prop. 4.34]{Mel93}, elements of the small calculus define bounded linear operators $C^\infty(X;\bOmegah) \too C^\infty(X;\bOmegah)$ and the symbol map and ellipticity are defined as for the standard pseudodifferential calculus, but this form of ellipticity is not enough to allow for invertibility up to compact errors. Again in \cite{Mel93}, the notion of full ellipticity is introduced and it is shown that there is a parametrix for fully elliptic operators which leads to a good Fredholm theory.

For the full calculus, there is so far no good Fredholm theory, but mapping properties between spaces of polyhomogeneous $b$-half densities and weighted $b$-Sobolev spaces have been established, see for instance \cite{GH14}.

\begin{define}[\textb-Sobolev Spaces] Let $\mathrm{dvol}_b$ be a smooth and non-vanishing $b$-density on $X$, $x$ a global boundary defining function and $\chi$ a cut-off function which is equal to $1$ near $x=0$ and equal to $0$ for $x > \eps$. Then, the $b$-Sobolev space of order $m \in \N$ and with weight $s \in \R$ is
 \[ x^s H_b^m(X,\mathrm{dvol}_b)
     = \myset{ x^s u }{ u \in H^m_{loc}(\mathring{X}),\,
                             (xD_x)^j(D_y)^\alpha \chi u \in L^2(M,\mathrm{dvol}_b)
                             \,\forall\, j+|\alpha| \le m } \;. \]
\end{define}
As is noted in \cite{GH14}, this definition does not depend on the choice of $\mathrm{dvol}_b$, $\chi$ and coordinates and can be generalised to $m \in \R$ by arguments of duality and interpolation.

\begin{thm}\label{APP.thm:full.mapping} Let $X$ be a compact manifold with boundary, $x$ be a global boundary defining function for $\partial X$, $P \in \Psi_b^{m,\calG}(X)$ and $\alpha$, $\beta \in \R$, $k \in \R$.
 \begin{enumerate}
  \item If $\calG(\lf) > \beta$, $\calG(\rf) > -\alpha$ and $\calG(\bff) \ge \beta - \alpha$, then $P$ is bounded as an operator
       \[ P : x^\alpha H_b^{k+m}(X,\mathrm{dvol}_b) \too x^\beta H_b^k(X,\mathrm{dvol}_b). \]
  \item If $m < 0$ and strict equality holds everywhere in item i), then $P$, acting as before, is compact.
  \item If $\calG(\rf) + E > 0$, then $P$ is bounded as an operator
       \[ P : \calA^{E}(X;\bOmegah) \too \calA^{P_\sharp(E)}(X;\bOmegah) \;, \]
      where $P_\sharp(E) = \calG(\lf) \bcup \big( \calG(\bff) + E \big)$.
 \end{enumerate}
\end{thm}


\clearpage
\bibliographystyle{myamsalpha}
\bibliography{half-space}

\end{document}